\documentclass[a4paper,draft]{ZbRad}
\usepackage{amssymb}
\usepackage{amscd} 
\usepackage{longtable}
\usepackage[pdftex]{graphicx}
\usepackage{amsmath}
\usepackage{mathtools}
\usepackage{todonotes}
\usepackage{amsthm}
\usepackage{amsfonts}
\usepackage{amssymb}
\usepackage[margin=1.0in]{geometry}
\usepackage{tikz-cd}
\usetikzlibrary{cd}
\usepackage{enumitem}

\theoremstyle{plain}
 \newtheorem{theorem}{Theorem}[section]
 \newtheorem{prop}[theorem]{Proposition}
 \newtheorem{lemma}[theorem]{Lemma}
 \newtheorem{obs}[theorem]{Observation}
 \newtheorem{cor}[theorem]{Corollary}
 \newtheorem{fact}[theorem]{Fact}
 \newtheorem{que}[theorem]{Question}
 \newtheorem{notation}[theorem]{Notation}
\theoremstyle{definition}
 \newtheorem{exa}[theorem]{Example}
 \newtheorem{rem}[theorem]{Remark}
 \newtheorem{defin}[theorem]{Definition}
\numberwithin{equation}{section}

\renewcommand{\leq}{\leqslant}
\renewcommand{\geq}{\geqslant}
\renewcommand{\setminus}{\smallsetminus}
\newfont{\TITf}{cmssdc10 scaled 1440}

\newcommand{\bb}[1]{\mathbb{#1}}
\renewcommand{\rm}[1]{\mathrm{#1}}
\renewcommand{\sf}[1]{\mathsf{#1}}

\newcommand{\bA}{\mathbf{A}}

\newcommand{\cA}{\mathcal{A}}

\newcommand{\bB}{\mathbf{B}}

\newcommand{\cB}{\mathcal{B}}

\newcommand{\bC}{\mathbf{C}}

\newcommand{\cC}{\mathcal{C}}

\newcommand{\bD}{\mathbf{D}}

\newcommand{\cD}{\mathcal{D}}

\newcommand{\bE}{\mathbf{E}}

\newcommand{\bF}{\mathbf{F}}

\newcommand{\cF}{\mathcal{F}}

\newcommand{\bK}{\mathbf{K}}

\newcommand{\cK}{\mathcal{K}}

\newcommand{\bbL}{\mathbb{L}}

\newcommand{\cL}{\mathcal{L}}

\newcommand{\bM}{\mathbf{M}}

\newcommand{\bbN}{\mathbb{N}}

\newcommand{\bbQ}{\mathbb{Q}}

\newcommand{\bbR}{\mathbb{R}}

\newcommand{\bS}{\mathbf{S}}
\newcommand{\bbS}{\mathbb{S}}

\newcommand{\bT}{\mathbf{T}}
\newcommand{\bbT}{\mathbb{T}}

\newcommand{\sft}{\mathsf{t}}

\newcommand{\bU}{\mathbf{U}}
\newcommand{\bbU}{\mathbb{U}}

\newcommand{\sfu}{\mathsf{u}}

\newcommand{\ol}{\overline}

\newcommand{\la}{\langle}
\newcommand{\ra}{\rangle}

\newcommand{\flim}{\mathrm{Flim}}
\newcommand{\age}{\mathrm{Age}}
\newcommand{\fin}[1]{[#1]^{{<}\omega}}
\newcommand{\fr}{Fra\"iss\'e }
\renewcommand{\phi}{\varphi}
\newcommand{\emb}{\mathrm{Emb}}
\newcommand{\aut}{\mathrm{Aut}}

\newcommand{\dom}{\mathrm{dom}}

\newcommand{\im}{\mathrm{Im}}

\newcommand{\circdots}{\circ\cdots\circ}

\newcommand{\cdnd}{\mathrm{CdNd}}

\newcommand{\is}{\mathrm{IS}}

\newcommand{\str}{\mathbf{Str}}

\newcommand{\ct}{\mathrm{CT}}
\newcommand{\aemb}{\mathrm{AEmb}}

\newcommand{\lex}{\prec_{\rm{lex}}}
\newcommand{\lexeq}{\preceq_{\rm{lex}}}
\def\-{\raisebox{.30pt}{-}}

\title[A survey on big Ramsey structures]{}
\author{Jan Hubi\v{c}ka and Andy Zucker}

\begin{document}
\setcounter{page}{1}

\vspace*{40mm}

\thispagestyle{empty}

{
\TITf\setlength{\parskip}{\smallskipamount}

\begin{center}
Jan Hubi\v{c}ka and Andy Zucker

\bigskip\bigskip\bigskip

{A survey on big Ramsey structures}

\end{center}
}

\vspace{6em}{\leftskip3em\rightskip3em

\emph{Abstract}.
    In recent years, there has been much progress in the field of structural Ramsey theory, in particular in the study of big Ramsey degrees. In all known examples of infinite structures with finite big Ramsey degrees, there is in fact a single expansion of the structure, called a big Ramsey structure, which correctly encodes the exact big Ramsey degrees of every finite substructure simultaneously. The first half of the article collects facts about this phenomenon that have appeared in the literature into a single cohesive framework, thus offering a conceptual survey of big Ramsey structures. We present some original results indicating that the standard methods of proving finite big Ramsey degrees automatically yield big Ramsey structures, often with desirable extra properties. The second half of the article is a survey in the more traditional sense, discussing numerous examples from the literature and showing how they fit into our framework. We also present some general results on how big Ramsey degrees are affected by expanding structures with unary functions. 



}

\newpage\thispagestyle{empty}
\maketitle
\tableofcontents

\section{Introduction}
\label{Sec:Intro}

We use standard set-theoretic notation. We identify a non-negative integer $k$ with the set $\{0,\ldots,k-1\}$, though we often write the latter for emphasis. Given sets $X$ and $Y$, a function $f\colon X\to Y$, and $S\subseteq X$, we write $f[S] = \{f(s): s\in S\}$. Given a set $X$ and cardinal $k$, we write $\binom{X}{k} = \{Y\subseteq X: |Y| = k\}$. 

The infinite Ramsey theorem \cite{Ramsey} states that for any $0< k, r< \omega$ and any coloring $\gamma\colon \binom{\omega}{k}\to r$, then there is $X\in \binom{\omega}{\omega}$ with $\gamma$ constant on $\binom{X}{k}$. Upon attempting to generalize this result to other countable, first-order structures, the situation becomes much more interesting. For instance, consider $\bbQ$ viewed as a linearly ordered set; a subset $X\subseteq \bbQ$ is \emph{non-scattered} if there is some order-preserving injection from $\bbQ$ into $X$. Sierpi\'nski in \cite{Sierpinski} constructed a coloring $\gamma_2\colon \binom{\bbQ}{2}\to 2$ such that whenever $X\subseteq \bbQ$ is non-scattered, then $\gamma_2[\binom{X}{2}] = 2$. Yet several decades later, Galvin \cite{Galvin} proved that this was worst possible; for any $r< \omega$ and $\gamma\colon \binom{\bbQ}{2}\to r$, there is a non-scattered $X\subseteq \bbQ$ with $|\gamma[\binom{X}{2}]| \leq 2$. By unpublished work of Laver (see \cite{Stevo_book}) and the thesis of Devlin \cite{Devlin}, a similar phenomenon happens for every $k$ -- there is a number $r_k< \omega$ such that both of the following happen:
    \begin{itemize}
        \item 
        There is $\gamma_k\colon \binom{\bbQ}{k}\to r_k$ such that whenever $X\subseteq \bbQ$ is non-scattered, then $\gamma_k[\binom{X}{k}] = r_k$.
        \item 
        For any $r< \omega$ and $\gamma\colon \binom{\bbQ}{k}\to r$, there is a non-scattered $X\subseteq \bbQ$ with $|\gamma[\binom{X}{k}]|\leq r_k$. 
    \end{itemize}
    Devlin actually shows something more; the colorings $\{\gamma_k: 2\leq k< \omega\}$ can actually be built simultaneously in a coherent fashion. Equivalently, there is an \emph{expansion} of the rational linear order such that the map sending a $k$-tuple from $\bbQ$ to the induced expansion of it is a valid choice of $\gamma_k$ as above.

    The number $r_k$ is called the \emph{big Ramsey degree} of the $k$-element linear order in $\bbQ$. In a similar fashion, one can define the big Ramsey degree of any finite substructure of an infinite structure (Definition~\ref{Def:BRD}) and ask which infinite structures have finite big Ramsey degrees. It so happens that in all known examples of infinite structures with finite big Ramsey degrees, we can in fact find a single expansion of the infinite structure which correctly encodes the exact big Ramsey degrees of every finite substructure simultaneously. Observing this, and motivated by questions in topological dynamics posed in \cite{KPT}, the second author in \cite{Zucker_BR_Dynamics} defined the notion of a \emph{big Ramsey structure}, an expansion of a given infinite structure which precisely encodes big Ramsey degrees. Various recent works \cite{BCDHKVZ, BCDHKVZ_poset, CDP_SDAP1, CDP_SDAP2, DZ} provide a wealth of new examples of big Ramsey structures and isolate extra desirable properties they might have, for example being \emph{recurrent} (Definition~\ref{Def:Props_of_expansions}). While a number of basic lemmas regarding big Ramsey structures appear in these works, the assumptions stated therein are always tailored to the specific situation at hand.

    The first half of this article collects the various properties of big Ramsey structures that have been considered in the literature and presents them in a single abstract, cohesive framework. In so doing, we are able to isolate exactly which assumptions are needed for various propositions to hold. In particular, while big Ramsey structures were first studied in the case that the un-expanded structure is a \emph{\fr structure} (Section~\ref{Sec:Background_BRD}), the definition was generalized in \cite{ACDMP} to arbitrary infinite structures, and many of the basic properties go through after dropping the \fr assumption, or even countability. In general, even if one is primarily interested in the big Ramsey properties of \fr structures, it becomes necessary to investigate structures which are not $\omega$-homogeneous. For instance, most \emph{big Ramsey structures} (Definition~\ref{Def:BRS}) cannot be $\omega$-homogeneous (Proposition~7.19 of \cite{Zucker_BR_Dynamics}). Furthermore, various common steps in the proofs of Ramsey theorems, such as adding a linear order in order type $\omega$, destroy $\omega$-homogeneity. As an application of our framework, we prove some original results, Theorems~\ref{Thm:Weak_Biembeddability} and \ref{Thm:Implies_Finite_Recurrent_BRS}, which show that the standard approaches to proving finite big Ramsey degree results always yield recurrent big Ramsey structures. The second half of this article gives a brief survey (we point to \cite{Dobrinen_ICM} for a more general survey on big Ramsey degrees), giving accounts of various examples of big Ramsey structures that have appeared in the literature and shows how applications of Theorems~\ref{Thm:Weak_Biembeddability} and \ref{Thm:Implies_Finite_Recurrent_BRS} can be used to derive the key features of these examples. In particular, one such application is a new elementary proof of a theorem of Laflamme, Sauer, and Vuksanovic \cite{Laflamme_S_V} used in characterising the big Ramsey degrees of the Rado graph.

\section{Background on big Ramsey degrees}
\label{Sec:Background_BRD}
 
	All structures considered in this paper, with the exception of Section~\ref{Sec:Functions}, are relational. \textbf{We fix once and for all} a relational language $\bbL$, a set of relation symbols $R$ each equipped with an \emph{arity} $0< n_R < \omega$. All languages discussed will be subsets of $\bbL$, and will typically be denoted by $\cL$, $\cL^*$, etc. When $\cL\subseteq \bbL$ only consists of unary and binary relations, we simply call $\cL$ \emph{binary}. An \emph{$\bbL$-structure} $\bA$ (or, since $\bbL$ is fixed, just \emph{structure}) is a set $A$ (the \emph{universe} or \emph{underlying set} of $\bA$) along with a distinguished subset $R^{\bA}\subseteq A^{n}$ for each $R\in \bbL$ of arity $n$. For $R\in \bbL$ of arity $1$ (i.e.\ \emph{unary} relation symbols), we can also write $R(\bA)\subseteq A$ in place of $R^\bA$. Unless indicated otherwise, we typically denote $\bbL$-structures in bold letters (possibly with other decoration) and use the un-bolded letter to denote the underlying set, i.e.\ $A, B, C$ are the underlying sets of $\bA$, $\bB^*$, $\bC'$, etc. A structure $\bA$ is \emph{finite}, \emph{countable}, \emph{countably infinite}, etc.\ iff $A$ is, and $\bA$ is \emph{enumerated} if $A = |A|$. Given a structure $\bA$, we let $\cL_\bA = \{R\in \bbL: R^\bA \neq \emptyset\}$, and given a class $\cK$ of structures, we let $\cL_\cK = \bigcup_{\bA\in \cK} \cL_\bA$. 
 
    In what follows, $\bA, \bB, \bK$, etc.\ denote structures. An \emph{embedding} $f\colon \bA\to \bB$ is an injection from $A$ to $B$ such that for every $R\in \bbL$ of arity $n$ and every $(a_0,\ldots,a_{n-1})\in A^n$, we have $(a_0,\ldots,a_{n-1})\in R^\bA$ iff $(f(a_0),\ldots,f(a_{n-1}))\in R^\bB$. Write $\emb(\bA, \bB)$ for the set of embeddings of $\bA$ into $\bB$; if $\bA = \bB$, we simply write $\emb(\bA)$; note that $\emb(\bA)$ is a monoid under composition. When $\bA$ is infinite, we typically equip $\emb(\bA)$ with the topology of pointwise convergence. We write $\aut(\bA)\subseteq \emb(\bA)$ for the bijective members of $\emb(\bA)$; this is the \emph{autmorphism group} of $\bA$. We say $\bA$ is a \emph{substructure} of $\bB$ if $A\subseteq B$ and the inclusion map is an embedding of $\bA$ into $\bB$. A \emph{copy} of $\bA$ in $\bB$ is the image of an embedding of $\bA$ into $\bB$, and we write $\binom{\bB}{\bA}$ for the set of copies of $\bA$ in $\bB$. We write $\bA\leq \bB$ iff $\emb(\bA, \bB)\neq\emptyset$ iff $\binom{\bB}{\bA}\neq \emptyset$. We say $\bA$ and $\bB$ are \emph{bi-embeddable} if both $\bA\leq \bB$ and $\bB\leq \bA$. We write $\age(\bK) = \{\bA: |A|< \omega \text{ and } \bA\leq \bK\}$.   

    A \emph{\fr class} of structures is a class $\cK$ of finite structures which is closed under isomorphism, countable up to isomorphism, contains arbitrarily large finite structures, and satisfies the following three key properties.
	\begin{itemize}
		\item 
		$\cK$ has the \emph{hereditary property} (HP): Whenever $\bB\in \cK$ and $\bA\leq \bB$, then $\bA\in \cK$.
		\item 
		$\cK$ has the \emph{joint embedding property} (JEP): Whenever $\bA, \bB \in \cK$, there is $\bC\in \cK$ with both $\bA\leq \bC$ and $\bB\leq \bC$.
		\item 
		$\cK$ has the \emph{amalgamation property} (AP): Whenever $\bA, \bB, \bC\in \cK$, $f\in \emb(\bA, \bB)$, and $g\in \emb(\bA, \bC)$, there are $\bD\in \cK$, $r\in \emb(\bB, \bD)$, and $s\in \emb(\bC, \bD)$ with $r\circ f = s\circ g$.
	\end{itemize}
	\fr \cite{fraisse_1954} proves that given a \fr class $\cK$, there is up to isomorphism a unique countably infinite structure $\bK$ with $\age(\bK) = \cK$ and with $\bK$ \emph{$\omega$-homogeneous}, i.e.\ for any finite $\bA\subseteq \bK$ and any $f\in \emb(\bA, \bK)$, there is $g\in \aut(\bK)$ with $g|_\bA = f$. This unique $\bK$ is called the \emph{\fr limit} of $\cK$ and is sometimes written as $\flim(\cK)$. Conversely, whenever $\bK$ is countably infinite and $\omega$-homogeneous (i.e.\ a \emph{\fr structure}), $\age(\bK)$ is a \fr class. 
 
	Given structures $\bA\leq \bB\leq \bC$ and positive integers $t< r$, we write
	$$\bC\longrightarrow (\bB)^\bA_{r, t}$$ 
	if for any $\chi\colon \emb(\bA, \bC)\to r$, there is $g\in \emb(\bB, \bC)$ with $|\chi[g\circ \emb(\bA, \bB)]|\leq t$; when $t = 1$, we omit the subscript. Historically, what was first considered was the above definitions, but coloring copies instead of embeddings. One can show (see Section 4 of \cite{Zucker_Metr_UMF}) that if $\bA\leq \bB\leq \bC$ are structures with $\aut(\bA)$ finite and $t< \omega$, then we have
    $$\forall r< \omega \left(\bC\xrightarrow{\mathrm{copy}} (\bB)^\bA_{r, t}\right) \Leftrightarrow \forall r< \omega \left(\bC\longrightarrow (\bB)^\bA_{r, t{\cdot}|\aut(\bA)|}\right).$$

    While the following concept is implicit in a number of earlier works \cite{Sierpinski, Galvin, Devlin, PS_1996, Sauer_directed_graphs}, the following definition was first isolated in \cite{KPT}.
	
	\begin{defin}  
    \label{Def:BRD}
		Let $\bK$ be an infinite structure,  and let $\bA\in \age(\bK)$. The \emph{big Ramsey degree} of $\bA$ in $\bK$, denoted $\rm{BRD}(\bA, \bK)$, is the least $t< \omega$ such that $\forall r< \omega\left(\bK\to (\bK)^\bA_{r, t}\right)$. By a standard color-fusing argument, this holds iff $\bK\to (\bK)^\bA_{t+1, t}$. If there is no $t< \omega$ for which this holds, we write $\rm{BRD}(\bA, \bK) = \infty$. We say that $\bK$ has \emph{finite big Ramsey degrees} if $\rm{BRD}(\bA, \bK)< \infty$ for every $\bA\in \age(\bK)$. We say that $\bA\in \age(\bK)$ is a \emph{big Ramsey object} of $\bK$ if $\rm{BRD}(\bA, \bK) = 1$. Let $\rm{BRO}(\bK)\subseteq \age(\bK)$ denote the class of big Ramsey objects of $\bK$. We say that $\bK$ \emph{satisfies the infinite Ramsey theorem} (IRT) if $\rm{BRO}(\bK) = \age(\bK)$. 

        Given a \fr class $\cK$ with limit $\bK$, we say that $\cK$ has \emph{finite big Ramsey degrees} if $\rm{BRD}(\bA, \bK)$ is finite for every $\bA\in \cK$, and we can write $\rm{BRD}(\bA, \cK)$ and $\rm{BRD}(\bA, \bK)$ interchangeably.
	\end{defin} 

\begin{rem}
    In many sources, including \cite{KPT}, $\rm{BRD}(\bA, \cK)$ is denoted by $T_\cK(\bA)$. Following \cite{BCDHKVZ}, we choose to change the notation due to the number of tree-like objects we will introduce using the letter T.
\end{rem}    

Thus the original infinite Ramsey theorem \cite{Ramsey} is equivalent to the statement that the structure $\la \omega, <\ra$ satisfies IRT. We note that by a theorem of Hjorth \cite{Hjorth}, whenever $\bK$ is a \fr structure with $\aut(\bK)$ non-trivial, then $\bK$ does not satisfy IRT.

It is natural to ask for big Ramsey degrees to be monotone. Proposition~\ref{Prop:Large_image_monotone_BRD} gives a natural extra assumption ensuring that this is the case.

\begin{defin}
    \label{Def:Unavoidable}
    Given an infinite structure $\bK$, $\bA\in \age(\bK)$, $S\subseteq \emb(\bA, \bK)$, and $\eta\in \emb(\bK)$, we write $\eta^{-1}(S):= \{f\in \emb(\bA, \bK): \eta\circ f\in S\}$. We call $S\subseteq \emb(\bA, \bK)$ \emph{large} if there is $\eta\in \emb(\bK)$ with $\eta^{-1}(S) = \emb(\bA, \bK)$. We call $S\subseteq \emb(\bA, \bK)$ \emph{unavoidable} (in some references \emph{persistent})  if $\emb(\bA, \bK)\setminus S$ is not large. We say that a coloring of $\emb(\bA, \bK)$ is \emph{unavoidable} if every color class is unavoidable. An \emph{avoidable} subset or coloring is one which is not unavoidable.
\end{defin}

\begin{prop}
    \label{Prop:Large_image_monotone_BRD}
    If $\bK$ is an infinite structure, $\bA\leq \bB\in \age(\bK)$, and there is $f\in \emb(\bA, \bB)$ so that $\emb(\bB, \bK)\circ f\subseteq \emb(\bA, \bK)$ is large, then $\rm{BRD}(\bA, \bK)\leq \rm{BRD}(\bB, \bK)$. 
\end{prop}

\begin{proof}
    Suppose $\rm{BRD}(\bB, \bK) = \ell <\infty$. Fix a finite coloring $\gamma\colon \emb(\bA, \bK)\to r$. Pick some $f\in \emb(\bA, \bB)$ and form the coloring $\gamma{\cdot} \hat{f}: \emb(\bB, \bK)\to r$ given by $\gamma{\cdot}\hat{f}(x) = \gamma(x\circ f)$. We may find $\eta\in \emb(\bK)$ so that 
    \begin{align*}
        &\left\vert \gamma{\cdot}\hat{f}\left[\eta\circ \emb\left(\bB, \bK\right)\right]\right\vert\leq \ell\\
        \Leftrightarrow  &\Big\vert\gamma\left[\eta\circ \emb\left(\bB, \bK\right)\circ f\right]\Big\vert\leq \ell.
    \end{align*}
    As $\emb(\bB, \bK)\circ f\subseteq \emb(\bA, \bK)$ is large by assumption, find $\theta\in \emb(\bK)$ with $\theta\circ \emb(\bA, \bK)\subseteq \emb(\bB, \bK)\circ f$. Hence $\left\vert\gamma\left[\eta\circ \theta\circ \emb\left(\bA, \bK\right)\right]\right\vert\leq \ell$.   
\end{proof}

Note that if $\bK$ is a countable structure, then $\bK$ is \fr iff for any $\bA\leq \bB\in \age(\bK)$ and $f\in \emb(\bA, \bB)$, we have $\emb(\bB, \bK)\circ f = \emb(\bA, \bK)$. Thus Proposition~\ref{Prop:Large_image_monotone_BRD} recovers the result from \cite{Zucker_BR_Dynamics} that BRDs in \fr structures are monotone.  Proposition~\ref{Prop:Large_image_monotone_BRD} is also implicitly used by Ma\v{s}ulovi\'c \cite{Masulovic_2020, Masulovic_2023} in some situations where $\bK$ is not $\omega$-homogeneous.

\section{Expansions and big Ramsey structures}
\label{Sec:Expansions}

Our next goal is to define reducts, expansions, and big Ramsey structures (Definitions~\ref{Def:Expansion} and \ref{Def:BRS}). Morally speaking, a reduct of a structure $\bK$ is obtained by forgetting some of the symbols in $\cL_\bK$, and an expansion is obtained by adding some interpretations of new relation symbols on top of $\bK$. We mention that our definition is a bit more strict than what some references allow. For instance, some references allow one to form a ``reduct'' by first adding some relations which are definable from $\bM$, then deleting some other relations; an example of this is forming a two-graph from a graph~\cite{Thomas1991}.

It is quite natural to consider adding relations to structures when considering colorings of embeddings. Given a structure $\bK$, $\bA\in \age(\bK)$, and a coloring $\gamma\colon \emb(\bA, \bK)\to r$, we can encode this coloring by adding $r$ new $|A|$-ary relations on top of $\bK$. In principle, to encode several colorings of $\emb(\bA, \bK)$ for varying $\bA\in \age(\bK)$, one would need to form several different expansions. The difficulty is this: Suppose $\bA\leq \bB\in \age(\bK)$ and $\gamma_\bA, \gamma_\bB$ are finite colorings of $\emb(\bA, \bK)$ and $\emb(\bB, \bK)$, respectively. If we attempt to na\"ively add both of the corresponding expansions to $\bK$ simultaneously, obtaining some expansion $\bK^*$, then the $\gamma_\bB$-color of some $g\in \emb(\bB, \bK)$ might no longer tell us exactly how $\bK^*$ looks on $\im(g)$, since $\bK^*|_{\im(g)}$ will also tell us information about $\gamma_\bA$. In situations where $\bK$ admits a big Ramsey structure, we will be able to encode several colorings witnessing lower bounds for BRDs simultaneously using a single expansion.

	\begin{notation}
		If $\bK$ is a structure, $S$ is a set, and $\eta\colon S\to K$ is an injection, then $\bK{\cdot}\eta$ denotes the unique structure on underlying set $S$ such that $\eta\in \emb(\bK{\cdot}\eta, \bK)$.

        If additionally $\bA\leq \bK$ and $\gamma\colon \emb(\bA, \bK)\to r$ is a coloring, we let $\gamma{\cdot}\eta\colon \emb(\bA,\allowbreak \bK{\cdot}\eta)$ be defined via $(\gamma{\cdot}\eta)(f) = \gamma(\eta\circ f)$.
	\end{notation}
	
	\begin{defin}
		\label{Def:Expansion}
		Given $\cL\subseteq \bbL$ and a structure $\bK^*$, the structure $\bK^*|_\cL$ is the structure with underlying set $K$ such that given $R\in \bbL$, we have 
		\begin{align*}
		R^{\bK^*|_\cL} = \begin{cases}
			\emptyset &\quad \text{if } R\not\in \cL\\
			R^{\bK^*} &\quad \text{if } R\in \cL.
		\end{cases}
	\end{align*}
		Given a pair of structures $\bK^*$ and $\bK$, both on underlying set $K$, we call $\bK$ a \emph{reduct of $\bK^*$} or $\bK^*$ an \emph{expansion of $\bK$} if $\bK^*|_{\cL_\bK} = \bK$. The notation $\bK^*/\bK$ indicates that $\bK^*$ is an expansion of $\bK$, and we call $\bK^*/\bK$ an expansion.

		Given an expansion $\bK^*/\bK$ and $\bA\leq \bK$, we write $\bK^*(\bA, \bK) = \{\bK^*{\cdot}f: f\in \emb(\bA, \bK)\}$. Thus $\bK^*(\bA, \bK)$ is the set of expansions of $\bA$ which embed into $\bK^*$. If $I$ is a set of structures all of which embed into $\bK$, we write $\bK^*(I, \bK):= \bigcup_{\bA\in I} \bK^*(\bA, \bK)$. \qed
    \end{defin}

	The next definition collects some properties that expansions might enjoy.
	
	\begin{defin}
		\label{Def:Props_of_expansions}
		Let $\bK^*/\bK$ be an expansion with $\bK$ infinite, and fix $I\subseteq \age(\bK)$. For any ``$I$-property'' defined below, when $I = \age(\bK)$, we simply refer to ``property'' (except for \emph{finitary}; see below). 
		\begin{enumerate}
			\item 
			We call $\bK^*/\bK$ \emph{$I$-precompact} if $\bK^*(\bA, \bK)$ is finite for every $\bA\in I$, see \cite{Lionel_2013}. When $\bK^*/\bK$ is precompact, we can form the compact space 
            $$X_{\bK^*/\bK}:= \{\bK'/\bK: \age(\bK')\subseteq \age(\bK^*)\}$$
			where the topology is given by declaring that $\bK'_n/\bK\to \bK'_\infty/\bK$ iff for every $\bA\in \age(\bK)$ and $f\in \emb(\bA, \bK)$, we eventually have $\bK'_n{\cdot}f = \bK'_\infty{\cdot}f$. The natural right action of $\emb(\bK)$ on $X_{\bK^*/\bK}$ is continuous, see \cite{KPT, Zucker_Metr_UMF, Zucker_BR_Dynamics}.
			\item
			We call $\bK^*/\bK$ \emph{$I$-finitary} if we have $|\bK^*(I, \bK)|<\omega$ (in particular, $I$ must be finite) and furthermore, for every $\bB\in \age(\bK)$, if $\bB'\neq \bB^*\in \bK^*(\bB, \bK)$, then for some $\bA\in I$ and $f\in \emb(\bA, \bB)$, we have $\bB'{\cdot}f\neq \bB^*{\cdot}f$. Thus $\bK^*$ is determined by how it looks on copies of members of $I$ in $\bK$. For this definition, omitting $I$ from the notation means that for some finite $I$, the expansion is $I$-finitary. 
			
			In particular, finitary expansions are precompact, and a structure is finitary iff it is equivalent to one in a finite relational language; for more on structures in finite relational languages and their finitary expansions, see \cite{BPT_2013}.
   
			\item
			Given $I\subseteq \age(\bK)$, we say $\bK^*/\bK$ is  \emph{$I$-unavoidable} if for any $\eta\in \emb(\bK)$, we have $(\bK^*{\cdot}\eta)(I, \bK) = \bK^*(I, \bK)$. Equivalently, this happens if for every $\bA\in I$ and $\bA^*\in \bK^*(\bA, \bK)$, $\emb(\bA^*, \bK^*)\subseteq \emb(\bA, \bK)$ is unavoidable. Note that $\bK^*/\bK$ is unavoidable iff for every $\eta\in \emb(\bK)$, we have $\age(\bK^*{\cdot}\eta) = \age(\bK^*)$. Note that always $\age(\bK^*{\cdot}\eta)\subseteq \age(\bK^*)$. 
			\item
            We call $\bK^*/\bK$ \emph{recurrent} if for any $\eta\in \emb(\bK)$, there is $\theta\in \emb(\bK)$ with $\eta\circ \theta\in \emb(\bK^*)$, i.e.\ if $\emb(\bK^*)$ meets every right ideal of $\emb(\bK)$. This property is called \emph{self-similar} in \cite{Masulovic_2021}. 
			
			In particular, recurrent expansions are unavoidable.

            \item 
            We call $\bK^*/\bK$ \emph{embedding faithful} if $\emb(\bK^*) = \emb(\bK)$. 
            
			\item 
			If $\bK'/\bK$ is another expansion of $\bK$ and $I\subseteq \age(\bK)$, then $\bK'/\bK$ is an \emph{$I$-factor} of $\bK^*/\bK$ if for any $\bA\in I$ and $f_0, f_1\in \emb(\bA, \bK)$, we have that $\bK^*{\cdot}f_0 = \bK^*{\cdot}f_1$ implies $\bK'{\cdot}f_0 = \bK'{\cdot}f_1$, and $\bK^*/\bK$ and $\bK'/\bK$ are \emph{$I$-equivalent} if each is an $I$-factor of the other.

		\end{enumerate}
    When $I = \{\bA\in \age(\bK): A = |A| \leq n\}$ for some $n< \omega$, we can write $n$ instead of $I$. Whenever any of the properties above is applied to just $\bK$, it means that $\bK$ viewed as an expansion over its underlying set has the property. 
	\end{defin}

We collect some basic observations about Definition~\ref{Def:Props_of_expansions}. 

    \begin{fact}
        \label{Fact:Exp_facts}
        Fix an infinite $\cL$-structure $\bK$ and $I\subseteq \age(\bK)$. 
	\begin{enumerate}
        \item 
        \label{Item:Exp_facts_unav_brd_lower_bound}
        If $\bK^*/\bK$ is $I$-unavoidable, then $\rm{BRD}(\bA, \bK)\geq |\bK^*(\bA, \bK)|$ for each $\bA\in I$. If $\eta\in \emb(\bK)$, then $(\bK^*{\cdot}\eta)/\bK$ is also $I$-unavoidable. 
 
		\item
        \label{Item:Exp_facts_precompact_to_unav}
		If $I$ is finite and $\bK^*/\bK$ is $I$-precompact, then there is $\eta\in \emb(\bK)$ such that $(\bK^*{\cdot}\eta)/\bK$ is $I$-unavoidable (pick $\eta\in \emb(\bK)$ minimizing the possible value of $|(\bK^*{\cdot}\eta)(I, \bK)|$).
  
		\item
        \label{Item:Exp_facts_shrink_finite}
		If $I$ is finite and $\bK^*/\bK$ is $I$-finitary, then every member of $X_{\bK^*/\bK}$ is $I$-finitary.
	\end{enumerate}
\end{fact}

	The following definition is the main subject of this survey. It first appeared in \cite{Zucker_BR_Dynamics} in the case that $\bK$ is \fr and for general $\bK$ in \cite{ACDMP}. Here we rephrase it slightly to take advantage of some of the vocabulary defined above.
	
	\begin{defin}[\cite{Zucker_BR_Dynamics}, \cite{ACDMP}]
		\label{Def:BRS}
		Given $\bK^*/\bK$ an expansion with $\bK$ infinite, we call $\bK^*/\bK$ a \emph{big Ramsey structure} (BRS) \emph{for $\bK$} if it is unavoidable and every $\bA\in \age(\bK)$ satisfies $|\bK^*(\bA, \bK)| = \rm{BRD}(\bA, \bK) < \infty$. We say that $\bK$ \emph{admits a BRS} if there is some expansion $\bK^*$ of $\bK$ with $\bK^*/\bK$ a BRS. A \fr class \emph{admits a BRS} if its \fr limit does.  
	\end{defin}

In particular, any BRS is precompact, and if $\bK$ admits a BRS, then $\bK$ has finite BRDs. We note that if $\bK^*/\bK$ is a BRS and $\eta\in \emb(\bK)$, then $(\bK^*{\cdot}\eta)/\bK$ is also a BRS. However, members of $X_{\bK^*/\bK}$ need not be BRSs.

The motivation for defining big Ramsey structures comes from topological dynamics and from analogy with what happens when considering \emph{small Ramsey degrees}. The term \emph{Ramsey degree} was introduced by Fouch\'e \cite{Fouche_1998} and popularized in \cite{KPT}; we add the prefix \emph{small} to better distinguish from big Ramsey degrees.   Given a \fr class with limit $\bK$, we say that $\bA\in \cK$ has \emph{small Ramsey degree} $t< \omega$ if $t$ is least such that for every $\bA\leq \bB\in \cK$ and every $r< \omega$, we have $\bK\to (\bB)^\bA_{r, t}$. Much like how unavoidable expansions provide lower bounds for big Ramsey degrees, ``syndetic'' expansions provide lower bounds for small Ramsey degrees. Sometimes called the \emph{order property} \cite{KPT} or the \emph{expansion property} \cite{Lionel_2013}, we say that $\bK^*/\bK$ is \emph{syndetic} if every $\bK'\in X_{\bK^*/\bK}$ satisfies $\age(\bK') = \age(\bK)$. Define a \emph{small Ramsey expansion} of $\bK$ to be any syndetic expansion $\bK^*/\bK$ which witnesses the exact small Ramsey degrees. The key difference with small Ramsey degrees is compactness; if $\cK$ has finite small Ramsey degrees, then automatically small Ramsey expansions exist \cite{Zucker_Metr_UMF, Lionel_2019}. Furthermore, if $\bK^*/\bK$ is a small Ramsey expansion, then \emph{every} member of $X_{\bK^*/\bK}$ is a small Ramsey expansion, and conversely, every small Ramsey expansion of $\bK$ is equivalent to some member of $X_{\bK^*/\bK}$. Since by a result of Ne\v{s}et\v{r}il and R\"odl \cite{NR_1977} $\age(\bK^*)$ is always a \fr class, one can simply choose $\bK^*$ to be the \fr limit, yielding a \fr Ramsey expansion of $\bK$. It turns out (see \cite{KPT, Zucker_Metr_UMF, Lionel_2013}) that $X_{\bK^*/\bK}$ is a dynamically meaningful object; it is the \emph{universal minimal flow} of the topological group $\aut(\bK)$. 

    By contrast, for big Ramsey degrees, the following important question is wide open.
    \begin{que}
        \label{Que:BRD_implies_BRS}
        Suppose $\cK$ is a \fr class with finite BRDs. Does $\cK$ admit a BRS?
    \end{que}
   If $\cK$ is a \fr class with limit $\bK$ and admitting a BRS $\bK^*$, one can form $X_{\bK^*/\bK}$ and ask about its properties as a dynamical object. The main result of \cite{Zucker_BR_Dynamics} is that $X_{\bK^*/\bK}$ is defined up to $\aut(\bK)$-flow isomorphism by a dynamical universal property much like the universal minimal flow. This has the following important consequence.

    \begin{fact}
        \label{Fact:Given_BRS_all_BRS}
        Suppose $\cK$ is a \fr class with limit $\bK$ and admitting a BRS $\bK^*$. Then if $\bK'/\bK$ is any other BRS, we have that $\bK'/\bK$ is equivalent (Definition~\ref{Def:Props_of_expansions}) to some member of $X_{\bK^*/\bK}$. In particular, if \emph{some} BRS for $\bK$ is finitary, then \emph{every} BRS for $\bK$ is finitary.
    \end{fact}

    However, by asking that our BRSs satisfy extra properties, we can obtain, up to equivalence and bi-embeddability, a canonical choice of BRS.

    \begin{prop}[see Proposition~2.4 of \cite{DZ}]
        \label{Prop:RecFinBRS_canonical}
        If  $\bK^*, \bK'$ are two expansions of $\bK$ with $\bK^*/\bK$ a BRS and $\bK'/\bK$ $I$-precompact for some finite $I\subseteq \age(\bK)$, then there is $\eta\in \emb(\bK)$ such that $(\bK'{\cdot}\eta)/\bK$ is an $I$-factor of $(\bK^*{\cdot}\eta)/\bK$. If $\bK^*/\bK$ is also recurrent, then there is $\eta\in \emb(\bK)$ such that $(\bK'{\cdot}\eta)/\bK$ is an $I$-factor of $\bK^*/\bK$.

        In particular, if $\bK$ admits a recurrent, finitary BRS, then given any two recurrent BRSs $\bK^*$ and $\bK'$ for $\bK$, there is $\eta\in \emb(\bK)$ with $(\bK^*{\cdot}\eta)/\bK)$ equivalent to $\bK'/\bK$. 
    \end{prop}

 Theorem~\ref{Thm:Weak_Biembeddability} gives an abstract account of the main approach for putting upper bounds on big Ramsey degrees.  Given a structure $\bK$ whose BRDs we are interested in, we can attempt to compare $\bK$ with some other structure $\bM^*$ whose BRDs we already know something about. The ``comparison'' is a weak form of bi-embeddability.  

\begin{defin}
    \label{Def:Weak_biembedding}
    Fix a class $\cK$ of finite structures along with structures $\bB$ and $\bC$. We say that a map $f\colon B\to C$ is a \emph{$\cK$-approximate embedding} from $\bB$ to $\bC$ if for each $\bA\in \cK$, we have $f\circ \emb(\bA, \bB)\subseteq \emb(\bA, \bC)$. Write $\emb_\cK(\bB, \bC)$ for the set of $\cK$-approximate embeddings from $\bB$ to $\bC$. When $\cK = \age(\bK)$, we can simply say $\bK$-approximate and write $\emb_\bK(\bB, \bC)$. In this case, note that for any (not necessarily finite) $\bA\leq \bK$, we have $f\circ \emb(\bA, \bB)\subseteq \emb(\bA, \bC)$.  
    
    Fix an infinite structure $\bK$. We call a tuple $(\bM, \phi, \psi)$ a \emph{weak bi-embedding} for $\bK$ if $\bM$ is an infinite structure, $\phi\in \emb(\bK, \bM)$, and $\psi\in \emb_\bK(\bM, \bK)$. 
\end{defin}

As an example, consider $\bK$ the rational linear order and $\bM$ the countable generic partial order (see Example~\ref{exa:poset}). Then letting $\phi\in \emb(\bK, \bM)$ and letting $\psi\colon M\to K$ be any embedding of some total linear extension of $\bM$ into $\bK$, then $(\bM, \phi, \psi)$ is a weak bi-embedding for $\bK$. 

\begin{theorem}
	\label{Thm:Weak_Biembeddability}
    Suppose $\bK$ is an infinite structure and $(\bM, \phi, \psi)$ is a weak bi-embedding for $\bK$. Then given an expansion $\bM^*/\bM$ and $\bA\in \age(\bK)$, we have $$\rm{BRD}(\bA, \bK)\leq \sum_{\bA^*\in (\bM^*{\cdot}\phi)(\bA, \bK)} \rm{BRD}(\bA^*, \bM^*). \eqno \qed$$
  
\end{theorem}

\begin{proof}
	The theorem statement is vacuous if $(\bM^*{\cdot}\phi)(\bA, \bK)$ is infinite, so assume it is finite, and write $(\bM^*{\cdot}\phi)(\bA, \bK) = \{\bA^*_k: k< n\}$. Fix $0< r< \omega$ and a coloring $\gamma\colon \emb(\bA, \bK)\to r$. If $k< n$ and if embeddings $\{\eta_i: i< k\}\subseteq \emb(\bM^*)$ have been determined, write $\xi_k = \eta_0\circdots \eta_{k-1}$, and consider the coloring $\gamma{\cdot}\psi{\circ}\xi_k \colon \emb(\bA^*_k, \bM^*)\to r$. Find $\eta_k\in \emb(\bM)$ with $|(\gamma{\cdot}\psi{\circ}\xi_k)[\eta_k{\circ}\emb(\bA^*_k,\allowbreak \bM^*)]|\leq \rm{BRD}(\bA^*_k, \bM^*)$, and put $\xi_{k+1} = \xi_k\circ \eta_k$. Once all of $\eta_0,\ldots, \eta_{n-1}$ have been determined, then we have $\psi\circ \xi_n \circ \phi\in \emb(\bK)$. To show that $|\gamma[\psi{\circ}\xi_n{\circ}\phi{\circ}\emb(\bA,\allowbreak \bK)]|\leq \sum_{k< n} \rm{BRD}(\bA^*_k, \bM)$, fix $f\in \emb(\bA, \bK)$. For some $k< n$, we have $\bM^*{\cdot}(\phi\circ f) = \bA^*_k$. For this $k< n$, we have $\eta_{k}\circdots\eta_{n-1}\circ \phi\circ f\in \eta_k\circ \emb(\bA^*_k, \bM^*)$, implying that $\gamma(\psi{\circ}\xi_n{\circ}\phi{\circ}f)\in (\gamma{\cdot}\psi{\circ}\xi_k)[\eta_k{\circ}\emb(\bA^*_k, \bM^*)]$. 
\end{proof}

\begin{cor}
    \label{Cor:Bi_embeddability_Ramsey}
    In the setting of Theorem~\ref{Thm:Weak_Biembeddability}, if $\age(\bM^*{\cdot}\phi)\subseteq \rm{BRO}(\bM^*)$, then $\rm{BRD}(\bA, \bK)\leq |(\bM^*{\cdot}\phi)(\bA, \bK)|$. 
\end{cor}

The next corollary is originally due to Ma\v{s}ulovi\'c and generalizes Theorem~4.3 of \cite{DB}. 

\begin{cor}[\cite{Masulovic_2021}, Theorems 6.1, 6.3]
    \label{Cor:Precompact_exp_BRDs}
    If $\bK^*/\bK$ is an expansion and $\bA\in \age(\bK)$, then 
    $$\rm{BRD}(\bA, \bK)\leq \sum_{\bA^*\in \bK^*(\bA, \bK)} \rm{BRD}(\bA^*, \bK^*).$$ 
    If $\bK^*/\bK$ is also recurrent, then we have equality. 
\end{cor}

\begin{proof}
    The $\leq$ direction follows from Theorem~\ref{Thm:Weak_Biembeddability} by considering the weak bi-embedding $(\bK, \rm{id}_K, \rm{id}_K)$. In the case $\bK^*/\bK$ is recurrent, towards showing $\geq$, we may assume $\rm{BRD}(\bA, \bK)$ is finite. Recurrence then tells us that $\bK^*(\bA, \bK)$ is finite; write $\bK^*(\bA, \bK) = \{\bA^*_k: k< n\}$. For each $k< n$, let $\gamma_k\colon \emb(\bA^*_k, \bK^*)\to r$ be a finite coloring, and let $\gamma\colon \emb(\bA, \bK)\to r$ denote their union. We may find $\eta\in \emb(\bK)$ with $|\gamma[\eta\circ \emb(\bA, \bK)]|\leq \rm{BRD}(\bA, \bK)$, and by recurrence, we may assume $\eta\in \emb(\bK^*)$. Hence $|\bigcup_{k< n} \gamma_k[\eta\circ \emb(\bA^*_k, \bK^*)]|\leq \rm{BRD}(\bA, \bK)$, giving the desired inequality.  
\end{proof}

\begin{cor}[\cite{DZ}, Proposition~2.7]
    \label{Cor:Recurrent_BRS_IRT}
    If $\bK^*/\bK$ is a recurrent BRS, then $\bK^*$ satisfies IRT. If $\bK^*/\bK$ is unavoidable, precompact, and $\bK^*$ satisfies IRT, then $\bK^*/\bK$ is a BRS.
\end{cor}

\begin{exa}
    \label{Exa:Countable_set}
    Let $\bK$ be the countable structure with no relations, i.e.\ a countable set $K$. If $\bK^*/\bK$ is an expansion which adds a linear order of order type $\omega$, then $\bK^*/\bK$ is recurrent, finitary, and $\bK^*$ satisfies IRT. It follows that $\bK^*/\bK$ is a BRS. By Fact~\ref{Fact:Given_BRS_all_BRS}, any other BRS for $\bK$ is equivalent to some member of the space $X_{\bK^*/\bK}$, which here is just the space of all linear orderings of $K$. It turns out (since for the class of finite sets, big and small Ramsey degrees coincide) that \emph{every} member of $X_{\bK^*/\bK}$ is a big Ramsey structure for $\bK$. The recurrent members of $X_{\bK^*/\bK}$ are exactly the  linear orders of order type $\omega$ or $\omega^*$ (the reverse of $\omega$). 

    In fact, almost all of the above discussion goes through when $\bK$ has no relations, $|K|$ is a weakly compact cardinal $\kappa$, and $\bK^*$ is a linear order of order type $\kappa$.
\end{exa}

Using Corollary~\ref{Cor:Precompact_exp_BRDs}, we can prove Proposition~\ref{Prop:BRS_over_recurrent}, a useful preservation property about how BRSs behave regarding ``expansions of expansions.''

\begin{lemma}
    \label{Lem:Exp_of_exp}
    If $\bK'/\bK$ and $\bK^*/\bK'$ are both recurrent expansions, then $\bK^*/\bK$ is also recurrent. If $\bK'/\bK$ is recurrent and $\bK^*/\bK'$ is unavoidable, then $\bK^*/\bK$ is also unavoidable.
\end{lemma}

\begin{proof}
    For both parts, fix $\eta\in \emb(\bK)$. As $\bK'/\bK$ is recurrent, find $\theta\in \emb(\bK)$ so that $\eta\circ \theta\in \emb(\bK')$. For the first part, as $\bK^*/\bK'$ is recurrent, find $\phi\in \emb(\bK')$ with $\eta\circ \theta\circ \phi\in \emb(\bK^*)$. As $\theta\circ \phi\in \emb(\bK)$ and $\eta\circ \theta\circ \phi\in \emb(\bK^*)$, we have that $\bK^*/\bK$ is recurrent. For the second part, we have $\age(\bK^*{\cdot}\eta\circ \theta) = \age(\bK^*)$. It follows that also $\age(\bK^*\cdot\eta) = \age(\bK^*)$.
\end{proof}

\begin{prop}
	\label{Prop:BRS_over_recurrent} 
	If $\bK'/\bK$ is a recurrent, precompact expansion and $\bK^*/\bK'$ is an expansion, then $\bK^*/\bK'$ is a BRS iff $\bK^*/\bK$ is a BRS.
\end{prop}

\begin{proof}

		First assume $\bK^*/\bK'$ is a BRS. Fix $\bA\in \age(\bK)$. Note that by Corollary~\ref{Cor:Precompact_exp_BRDs} and since $\bK^*/\bK'$ is a BRS, we have 
		$$\rm{BRD}(\bA, \bK) = \sum_{\bA'\in \bK'(\bA, \bK)} \rm{BRD}(\bA', \bK') = \sum_{\bA'\in \bK'(\bA, \bK)}|\bK^*(\bA', \bK')| = |\bK^*(\bA, \bK)|.$$
		By Lemma~\ref{Lem:Exp_of_exp}, $\bK^*/\bK$ is unavoidable, hence a BRS by the above equation. 
  
        Now assume $\bK^*/\bK$ is a BRS. As $\bK^*/\bK$ is unavoidable, so is $\bK^*/\bK'$, simply because $\emb(\bK')\subseteq \emb(\bK)$. We now verify that for every $\bA'\in \age(\bK')$ that $\rm{BRD}(\bA', \bK')= |\bK^*(\bA', \bK')|$. Unavoidability of $\bK^*/\bK'$ gives $\geq$. To get equality, write $\bA = \bA'|_{\cL_\bK}$. Corollary~\ref{Cor:Precompact_exp_BRDs} gives us
        $$\rm{BRD}(\bA, \bK) = \sum_{\bA'\in \bK'(\bA, \bK)} \rm{BRD}(\bA', \bK'),$$
        while the assumption that $\bK^*/\bK$ is a BRS gives
        $$\rm{BRD}(\bA, \bK)  = |\bK^*(\bA, \bK)| = \sum_{\bA'\in \bK'(\bA, \bK)}|\bK^*(\bA', \bK')|.$$
        Thus $\rm{BRD}(\bA', \bK') = |\bK^*(\bA', \bK')|$ for every $\bA'\in \bK'(\bA, \bK)$. Hence $\bK^*/\bK'$ is a BRS.
\end{proof}

Proposition~\ref{Prop:BRS_over_recurrent} explains a common step in almost all big Ramsey arguments on \fr limits of strong amalgamation classes, namely that of fixing an enumeration of the structure. Recall that a \fr class $\cK$ with limit $\bK$ has \emph{strong amalgamation} iff for each $F\in \fin{K}$ and $a\not\in F$, and letting $G = \aut(\bK)$, we have $\{g(a): g\in \rm{Stab}_G(F)\}$ infinite. Compare the following to Theorem~4.1 of \cite{Masulovic_2020}.

\begin{cor}
    \label{Cor:StrongAmalg}
    Let $\bK$ be the \fr limit of a strong amalgamation class. Then if $\bK'/\bK$ is an expansion adding a linear order of order type $\omega$, then $\bK'/\bK$ is recurrent; in fact, all such expansions are bi-embeddable. Additionally, if $\age(\bK)$ only contains finitely many structures on underlying set $2$, then $\bK'/\bK$ is finitary, and $\bK$ admits a BRS iff $\bK'$ does.
\end{cor}

\begin{proof}
    Given any two linear orders $<_0$ and $<_1$ of $K$ in order type $\omega$, we produce an embedding $\phi\colon \la \bK, <_1\ra \to \la \bK, <_2\ra$ inductively as follows. Let $\{a_n: n< \omega\}$ list $K$ in $<_0$-order. Let $\phi|_{\{a_0\}}$ be any partial embedding. If $\phi|_{\{a_k: k< n\}}$ has been produced for some $n> 0$, the strong amalgamation assumption ensures that in $\bK$, there are infinitely many vertices sharing the type of $a_n$ over $\{a_k: k< n\}$. Thus we can pick such a vertex $<_1$-above $a_{n-1}$ for $\phi(a_n)$.

    The statement about obtaining a finitary expansion follows from the definition, and the last statement follows from Proposition~\ref{Prop:BRS_over_recurrent}.
\end{proof}

If we can satisfy the assumptions of Theorem~\ref{Thm:Weak_Biembeddability} in a more uniform fashion, we obtain a much stronger conclusion.

\begin{theorem}
	\label{Thm:Implies_Finite_Recurrent_BRS}
	Fix an infinite structure $\bK$, and suppose that $(\bM, \phi, \psi)$ is a weak bi-embedding for $\bK$ such that for some finite $I\subseteq \age(\bK)$ and expansion $\bM^*/\bM$, the following hold:
	\begin{enumerate}
		\item
		$(\bM^*{\cdot}\phi)/\bK$ is $I$-finitary and $I$-unavoidable, 
        \item
        $(\bM^*{\cdot}\phi)(I, \bK)\subseteq \rm{BRO}(\bM^*)$.
	\end{enumerate}
	Then $(\bM^*{\cdot}\phi)/\bK$ is a finitary, recurrent BRS.
\end{theorem}

\begin{proof}
	To simplify notation, write $\bK^* = \bM^*{\cdot}\phi$. By Theorem~\ref{Thm:Weak_Biembeddability}, we have that $\rm{BRD}(\bA, \bK)\leq |\bK^*(\bA, \bK)|$ for every $\bA\in \age(\bK)$. To get the reverse inequalities, it suffices to show that $\bK^*/\bK$ is recurrent, which will show that $\bK^*/\bK$ is a finitary, recurrent BRS, finishing the proof.
	
	Fix $\theta\in \emb(\bK)$, towards showing that $\emb(\bK^*)\cap (\theta\circ \emb(\bK)) \neq\emptyset$. For each $\bA\in I$ and $\bA^*\in \bK^*(\bA, \bK)$, consider the coloring $\gamma_{\bA^*}\colon \emb(\bA^*, \bM^*)\to \bK^*(\bA, \bK)$ given by $\gamma_{\bA^*}(f) = \bK^*{\cdot}(\theta{\circ}\psi{\circ}f)$. By item 2, we can find $\zeta\in \emb(\bM^*)$ so that all of the colorings $\gamma_{\bA^*}{\cdot}\zeta$ are constant. In particular, consider the map $\psi\circ \zeta\circ \phi\in \emb(\bK)$. We claim that $\theta\circ \psi\circ \zeta\circ \phi$ is ``almost'' in $\emb(\bK^*)$, in the sense that whenever $\bA\in I$ and $f_0, f_1\in \emb(\bA, \bK)$, then $\bK^*{\cdot}f_0 = \bK^*{\cdot}f_1$ iff $\bK^*{\cdot}(\theta\circ \psi\circ \zeta\circ \phi){\cdot}f_0 = \bK^*{\cdot}(\theta\circ \psi\circ\zeta\circ \phi){\cdot}f_1$. The forward direction follows since $\gamma_{\bA^*}{\cdot}\zeta$ is monochromatic for each $\bA\in I$ and $\bA^*\in \bK^*(\bA, \bK)$. The reverse direction follows since $\bK^*/\bK$ is $I$-unavoidable. Thus the map $\theta\circ \psi\circ \zeta\circ \phi$ induces a permutation of each of the finite sets $\bK^*(\bA, \bK)$ for $\bA\in I$. It follows that for a suitably large power $n$ of $\theta\circ \psi\circ \zeta\circ \phi$, the corresponding permutation becomes the identity. Since $\bK^*/\bK$ is $I$-finitary, it follows that $(\theta\circ \psi\circ \zeta \circ \phi)^n \in \emb(\bK^*)\cap (\theta\circ \emb(\bK))$ as desired.	            
\end{proof}

\begin{cor}
    \label{Cor:IRT_finitary_recurrent}
    If $\bK$ is an infinite structure, $I\subseteq \age(\bK)$ is finite, $\bK^*/\bK$ is $I$-finitary and $I$-unavoidable, and $\bK^*$ satisfies IRT, then $\bK^*/\bK$ is recurrent (and hence a BRS). In particular, any BRS which is finitary and satisfies IRT is recurrent.
\end{cor}

Theorem~\ref{Thm:Implies_Finite_Recurrent_BRS} and Corollaries~\ref{Cor:Bi_embeddability_Ramsey} and \ref{Cor:IRT_finitary_recurrent} make it important to be able to produce structures that satisfy large fragments of IRT. This is often done by first defining a structure which is relatively easy to understand, then expanding it as needed in an embedding faithful way to enlarge the class of big Ramsey objects. This idea is captured by the following definition, whose various parts are referred to in the literature as the \emph{envelope}, \emph{shape}, or \emph{embedding type} of a given embedding.

\begin{defin}
\label{Def:Envelope_Expansion}
    Given a set $\cC$ of finite structures and another finite structure $\bA$, a \emph{$\cC$-extension of $\bA$} is a pair $(f, \bB)$ with $f\in \emb(\bA, \bB)$ and $\bB\in \cC$. Let $\rm{Ext}_\cC(\bA)$ denote the collection of such extensions. We equip $\rm{Ext}_\cC(\bA)$ with a partial order $\leq$, where $(f_0, \bB_0)\leq (f_1, \bB_1)$ iff there is $g\in \emb(\bB_0, \bB_1)$ with $f_1 = g\circ f_0$. A \emph{$\cC$-type} of $\bA$ is any $\leq$-minimal member of $\rm{Ext}_\cC(\bA)$, and we denote these by $\rm{Tp}_\cC(\bA)$. 

    Now suppose $\bM$ is an infinite structure with $\bA\in \age(\bM)$ and $\cC\subseteq \age(\bM)$. Given $(f, \bB)\in \rm{Ext}_\cC(\bA)$ and $g\in \emb(\bA, \bM)$, we say that $g$ \emph{extends to} $(f, \bB)$ if there is $h\in \emb(\bB, \bM)$ with $g = h\circ f$. We say that $\bM$ \emph{admits $\cC$-envelopes} if for any $\bA\in \age(\bM)$ and any $g\in\emb(\bA, \bM)$, $g$ extends to a unique $\cC$-type up to isomorphism, which we call the \emph{$\cC$-shape of $g$ in $\bM$} and denote by $\rm{Shp}_{\cC, \bM}(g)$. In this case, we define the expansion $\bM^{*\cC}$ by adding for each enumerated $\bA\in \age(\bM)$ and each $(f, \bB)\in \rm{Tp}_\cC(\bA)$ an $A$-ary relation $R_{(f, \bB)}$, where given $(a_0,\ldots, a_{A-1})\in M^A$, we have that $R_{(f, \bB)}^{\bM^{*\cC}}(a_0,\ldots, a_{A-1})$ holds iff the map $i\to a_i$ is in $\emb(\bA, \bM)$ and realizes $(f, \bB)$. We note that $\bM^{*\cC}/\bM$ is embedding faithful. \qed
\end{defin}

\begin{prop}
    \label{Prop:Envelopes}
    Suppose $\bM$ is an infinite structure and $\cC\subseteq \rm{BRO}(\bM)$ is such that $\bM$ admits $\cC$-envelopes. Then $\bM^{*\cC}$ satisfies IRT.  
\end{prop}

\begin{proof}
    This follows immediately from Proposition~\ref{Prop:Large_image_monotone_BRD}.
\end{proof}

\section{Applications using Milliken's theorem}
\label{Sec:Milliken}

We now turn to discussing how Theorem~\ref{Thm:Implies_Finite_Recurrent_BRS} and Corollaries~\ref{Cor:Bi_embeddability_Ramsey} and \ref{Cor:IRT_finitary_recurrent} are used implicitly in the literature to give bounds on big Ramsey degrees. We also use Theorem~\ref{Thm:Implies_Finite_Recurrent_BRS} to give some new proofs that various structures admit finitary recurrent BRSs.

To do this, we first need to discuss the main source of structures which satisfy IRT, which are typically trees equipped with various ``strong'' notions of embedding. Conventions about trees vary from reference to reference; as such, we also make various conventional choices in this work which are given in the next definition.

\begin{defin}
    \label{Def:tree}
    A \emph{level tree} (or just \emph{tree} in this survey, as all trees we mention will be level trees) is a structure $\bT:= \la T, \sqsubseteq^\bT, \wedge^\bT, \leq_{ht}^\bT\ra$, where
    \begin{enumerate}
        \item 
        $\sqsubseteq^\bT$ is a partial ordering with the property that for every $t\in T$, $\rm{Pred}_\bT(t):= \{s\in T: s\sqsubseteq^\bT t\}$ is linearly ordered and finite.
        \item 
        The \emph{$\bT$-height} of $t\in T$ is the number $\rm{ht}_\bT(t) = |\rm{Pred}_\bT(t)\setminus \{t\}|< \omega$, and given $m< \omega$, we write $\bT(m) = \{t\in T: \rm{ht}_\bT(t) = m\}$, similarly for $\bT({<}m)$, etc. The relation $s \leq_{ht}^\bT t$ holds exactly when $\rm{ht}_\bT(s)\leq \rm{ht}_\bT(t)$. 
        \item 
        The \emph{$\bT$-meet} of $s, t\in T$, if it exists, is the $\sqsubseteq^\bT$-largest $x\in T$ with $x\sqsubseteq^\bT s$ and $x\sqsubseteq^\bT t$; if such $x$ exists, we denote it by $s\wedge^\bT t$. Formally, $\wedge^\bT$ is a $3$-ary relational symbol, with $\wedge^\bT(s, t, x)$ iff $x = s\wedge^\bT t$, but we often write it as a binary partial function.
    \end{enumerate}
    We collect some other notation pertaining to trees.
    \begin{itemize}
        \item 
        If $t\in \bT(m)$, we write $\is_\bT(t):= \{u\in \bT(m+1): t\sqsubseteq^\bT u\}$ for the set of \emph{immediate successors} of $t$ in $\bT$.
        \item 
        If $t\in T$ and $m\leq \rm{ht}_\bT(t)$, we let $t|_m^\bT$ denote the unique member of $\rm{Pred}_\bT(t)$ in $\bT(m)$.
    \end{itemize}

    A \emph{subtree} of $\bT$ is any substructure $\bS\subseteq \bT$ which is a tree in the above sense. In particular, $S\subseteq T$ induces a subtree of $\bT$ iff $S$ is closed under $\wedge^\bT$ (closure under meets) and whenever $s_0, s_1\in S$ satisfy $s_0\leq_{ht}^\bT s_1$, we have $s_1|^\bT_{\rm{ht}_\bT(s_0)}\in S$ (closure under ``levels''). For arbitrary $S\subseteq T$, the \emph{ML-closure} (for ``meet-and-level'') of $S$ in $\bT$, denoted $\ol{S}^\bT$, is the smallest subtree of $\bT$ whose underlying set contains $S$. If $\bT^*/\bT$ is an expansion, we write $\ol{S}^{\bT^*}$ for the corresponding expansion of $\ol{S}^\bT$. 
\end{defin}

When $\bT$ is understood from context, we often omit $\bT$ as a subscript/superscript.   

A common example of a tree is $k^{<\omega}$ for some $k< \omega$ with its usual tree order, where given $s, t\in k^{<\omega}$, we put $s \sqsubseteq t$ iff $\dom(s)\leq \dom(t)$ and $t|_{\dom(s)} = s$. Note that in this tree, we have $\rm{ht}(t) = \dom(t)$. However, when referring to $k^{<\omega}$ as a tree, we typically add even more relations. We form the structure  
$$\bT_k:= \la k^{{<}\omega}, \sqsubseteq, \wedge, \leq_{ht}, \lexeq, (R_i)_{0< i< k}\ra.$$
Given $s, t\in k^{<\omega}$ with $\dom(s)< \dom(t)$, the \emph{passing number} of $t$ at $s$ is the number $t(\dom(s))< k$. To capture passing numbers, we add for each $0< i< k$ a binary relation $R_i$ so that $R_i(s, t)$ holds iff either $\rm{ht}(s) < \rm{ht}(t)$ and $t(\rm{ht}(s)) = i$ or vice versa. The lexicographic order $\lexeq$ is defined in the usual way; we note that $\lexeq$ is definable from the other relations, but it is helpful to add anyways (see Example~\ref{Exa:Rationals}). We can expand further by adding a unary predicate $U$ to nodes of level $< m$, and we write $\bT_{k, m}$ for this structure. We can now state a consequence of Milliken's tree theorem which is one of the linchpins of structural Ramsey theory.

\begin{theorem}[\cite{Milliken}]
    \label{Thm:Milliken}
	For any $k, m< \omega$, we have $\{\bT_{k, m}({<}n): m\leq n< \omega\}\subseteq \rm{BRO}(\bT_{k, m})$.
\end{theorem}

For most applications, one takes $m = 0$ (i.e.\ just $\bT_k$), but we will use the more general theorem in an application of Theorem~\ref{Thm:Implies_Finite_Recurrent_BRS}.

Theorem~\ref{Thm:Milliken}, which shows that $\rm{BRO}(\bT_{k, m})\subseteq \age(\bT_{k, m})$ is upwards cofinal, is insufficient to conclude that $\bT_{k, m}$ satisfies IRT. Given some $\bA\in \age(\bT_{k, m})$, there might not be $n< \omega$ and $f\in \emb(\bA, \bT_{k, m}({<}n))$ for which Proposition~\ref{Prop:Large_image_monotone_BRD} applies. To get around this, we expand the structure using Definition~\ref{Def:Envelope_Expansion}. We then have the following (see \cite{Stevo_book}).
\begin{theorem}
    \label{Thm:Milliken_Envelopes}
    For any $k, m< \omega$ and writing $\cC = \{\bT_{k, m}({<}n): m\leq n< \omega\}$, $\bT_{k, m}$ admits $\cC$-envelopes. Thus $\bT^{*\cC}_{k, m}$ satisfies IRT. Furthermore, $\bT^{*\cC}_{k, m}/\bT_{k, m}$ is $4$-finitary.
\end{theorem}

\begin{proof}
    We refer to \cite{Stevo_book} for most of the argument, but briefly discuss why the corresponding expansion is $4$-finitary. This follows from the observation that to determine the $\cC$-shape of some $g\in \emb(\bA, \bT_{k, m})$, it is enough to be able to describe the relative heights of nodes in the $\wedge$-closure of $\im(g)$, which we denote by $X$. As any node in $X$ can be described using at most two nodes in $\im(g)$, we can compare the relative heights of two nodes in $X$ by looking at a subset of $\im(g)$ of size at most $4$.
\end{proof}

\begin{exa}
    \label{Exa:Rationals}
    Consider the rational linear order $\la \bbQ, \leq\ra$. Sierpi\'nski \cite{Sierpinski} was the first to consider the Ramsey theory of this structure and constructed a coloring showing that $\rm{BRD}(2, \bbQ)\geq 2$ (Sierpi\'nski's coloring takes advantage of the fact that an expansion of $\bbQ$ adding a new linear order in order type $\omega$ is recurrent). Several decades later, Galvin \cite{Galvin} proved that in fact $\rm{BRD}(2, \bbQ) = 2$. Laver in unpublished work proved that $\bbQ$ has finite BRDs, more-or-less by proving parts of what would become Milliken's theorem. Devlin \cite{Devlin} then managed to characterise the exact BRDs for $\bbQ$, and his characterisation shows that $\bbQ$ admits a BRS. 

    Let us write $\bK$ for the rational linear order on underlying set $2^{<\omega}$ and with order given by $\prec_{lex}$. Then $\bT_2^*$ is an expansion of $\bK$. To show that $\bK$ admits a finitary, recurrent BRS using Theorem~\ref{Thm:Implies_Finite_Recurrent_BRS}, we find $\phi\in \emb(\bK, \bK)$ such that $(\bT_2^*{\cdot}\phi)/\bK$ is $4$-finitary and $4$-unavoidable. Such a $\phi$ can be described via objects which in this survey we call \emph{Devlin trees}, which are closely related to \emph{Joyce trees} (see \cite{ACDMP}). A \emph{Devlin tree} is a subtree $\bS\subseteq \bT_2^*$ such that for each $m< \omega$ with $\bS(m)\neq \emptyset$, exactly one of the following happens.
    \begin{enumerate}
        \item 
        There is exactly one $t\in \bS(m)$ with $|\is_\bS(t)| = 2$. We call $t$ a \emph{splitting node} of $\bS$. For each $s\in \bS(m)\setminus \{t\}$, we have $\is_\bS(s) = \{s'\}$ for some $s'$ with $\neg R_1(s, s')$ (i.e.\ passing number $0$). 
        \item 
        There is exactly one $t\in \bS(m)$ with neither $t^\frown 0$ nor $t^\frown 1\in \bS(m+1)$, i.e.\ $t$ is a terminal node in $\bS$. We call $t$ a \emph{coding node} of $\bS$. For each $s\in \bS(m)\setminus \{t\}$, we have $\is_\bS(s) = \{s'\}$ for some $s'$ with $\neg R_1(s, s')$ (i.e.\ passing number $0$).
    \end{enumerate}
    When $\bS$ is infinite, we additionally demand that any $s\in S$ can be $\sqsubseteq$-extended to a coding node of $\bS$.
    
    We sometimes expand $\bS$ by adding a unary predicate $\cdnd$ to label the coding nodes, obtaining a \emph{Devlin coding tree}. These are examples of the \emph{diagonal coding trees} defined in \cite{CDP_SDAP1} (see Example~\ref{Exa:CDP}). The \emph{structure coded by} $\bS$, denoted $\str(\bS)$ is the linear order  $\bK|_{\cdnd(\bS)}$. One can show (see \cite{Stevo_book}) that there is $\phi\in \emb(\bK)$ such that $\ol{\im(\phi)}^{\bT_2^*}$ is a Devlin tree whose coding nodes are exactly $\im(\phi)$. To apply Theorem~\ref{Thm:Implies_Finite_Recurrent_BRS}, one only needs to show that $(\bT_2^*{\cdot}\phi)/\bK$ is $4$-unavoidable; however, it is not too hard to show directly that any two Devlin coding trees which code a rational order are bi-embeddable \cite{Stevo_book}. 
\end{exa}

\begin{exa}
	\label{exa:2orders}
	This example was obtained jointly with Rivers Chen. Consider the class $\cK$ of finite structures equipped with two independent linear orders $\preceq_0$ and $\preceq_1$. One can think of the \fr limit $\bK$ as the ``rational plane.'' In particular, if we view $K\subseteq \bbR^2$ and the two orders coming from the usual ordering on the $x$ and $y$ coordinates, respectively, then $K$ must have the property that any two distinct points have distinct $x$-coordinates and distinct $y$-coordinates. Homogeneous structures with 2 linear orders can be seen as permutations and have been classified by Cameron~\cite{cameron2002homogeneous}. Homogeneous structures with $n$ linear orders, $n\geq 2$ are sometimes called \emph{$n$-dimensional permutations} and are classified by Braunfeld~\cite{braunfeld2020classification}.

    We will use an instance of the product form of Milliken's theorem. Let $$\bT_2\otimes\bT_2 = \la (2\times 2)^{{<}\omega}, (\sqsubseteq_i)_{i< 2}, (\wedge_i)_{i< 2}, \leq_{ht}, (\preceq_i)_{i< 2}, (R_{i, j})_{0< i, j< 2}\ra$$ 
    where the relations $\sqsubseteq_i$, $\wedge_i$, and $\preceq_i$ correspond to the structure $\bT_2$ in each of the two coordinate projections, and $R_{i, j}$ now encodes a ``2-dimensional'' passing number. In particular, the orders $\preceq_0$ and $\preceq_1$ are interpreted as the lexicographic orders in each coordinate. Much as in Theorem~\ref{Thm:Milliken}, we have that each finite substructure of the form $(\bT_2\otimes \bT_2)({<}n)\cong \bT_2({<}n)\otimes \bT_2({<}n)$ is in $\rm{BRO}(\bT_2\otimes \bT_2)$; indeed, one can see this as an instance of Theorem~\ref{Thm:Milliken} for the tree $\bT_{2, 1}$. As in Theorem~\ref{Thm:Milliken_Envelopes}, upon setting $\cC = \{(\bT_2\otimes \bT_2)({<}n): n< \omega\}$, then $\bT_2\otimes \bT_2$ admits $\cC$-envelopes, and the expansion $(\bT_2\otimes \bT_2)^{*\cC}/(\bT_2\otimes \bT_2)$ is $4$-finitary for almost the exact same reason.

    Let us write $\bM$ for $(\bT_2\otimes \bT_2)|_{\{\preceq_0, \preceq_1\}}$, also setting $M = (2\times 2)^{{<}\omega}$. So $\bM$ is isomorphic to $\bbQ\times \bbQ$ equipped with the usual linear order in each coordinate, which we refer to as $x$ and $y$ coordinates. Note that $\age(\bK)\subsetneq \age(\bM)$, since it is possible for distinct points in $M$ to have the same $x$-coordinate or the same $y$-coordinate. However, $\emb(\bK, \bM)\neq \emptyset$; indeed, given any finite $\bA\subseteq \bK$ and $f\in \emb(\bA, \bM)$, one can extend $f$ to a member of $\emb(\bK, \bM)$ by adding points to $\bA$ one at a time. One can also show that $\emb_\bK(\bM, \bK)\neq \emptyset$; in words, there is a map $\psi\colon M\to K$ which is ``correct'' on subsets of $M$ where distinct points have different $x$ and $y$ coordinates, but which ``breaks ties'' given any pair of points from $M$ with a common $x$ or $y$ coordinate. To build such a $\psi$, one can show that given a finite $\bB\subseteq \bM$ and $f\in \emb_\bK(\bB, \bK)$, then $f$ extends to a member of $\emb_\bK(\bM, \bK)$ by adding points to $\bB$ one at a time.

    Set $\bM^* = (\bT_2\otimes\bT_2)^{*\cC}$. To apply Theorem~\ref{Thm:Implies_Finite_Recurrent_BRS}, one needs to find $\phi\in \emb(\bK, \bM)$ so that $(\bM^*{\cdot}\phi)/\bK$ is $4$-finitary and $4$-unavoidable. We can describe such a $\phi$ using objects which are very similar to Devlin trees; in what follows, we drop some of the structural formalism to ease notation. We view $2\times 2^{{<}\omega}$ as two copies of $2^{{<}\omega}$ side by side. Given $S\subseteq 2\times 2^{{<}\omega}$, we write $S_i = S\cap (\{i\}\times 2^{{<}\omega})$, $S(n) = S\cap (2\times 2^n)$, and $S_i(n) = S_i\cap S(n)$. Given $t = (i, s)\in 2\times 2^{n}$ and $j< 2$, we write $t^\frown j = (i, s^\frown j)\in 2\times 2^{n+1}$. A \emph{product Devlin tree} is a subset $S\subseteq 2\times 2^{{<}\omega}$ closed under initial segments, with both $S_0$ and $S_1$ non-empty, with $S_0(n) = \emptyset$ iff $S_1(n) = \emptyset$, and such that for each $n$ with $S(n)\neq \emptyset$, exactly one of the following.
    \begin{enumerate}
        \item 
        There is exactly one $t\in S(n)$ which \emph{splits}, i.e.\ with both $t^\frown 0$ and $t^\frown 1\in S$. For $s\in S(n)\setminus\{t\}$, we have $s^\frown 0\in S(n+1)$.
        \item 
        There are exactly one $t_0\in S_0(n)$ and exactly one $t_1\in S_1(n)$ which are terminal. We call $(t_0, t_1)$ a \emph{product coding node} of $S$. For $s\in S(n)\setminus \{t_0, t_1\}$, we have $s^\frown 0\in S(n+1)$. 
    \end{enumerate}
    If $S$ is infinite, we additionally demand that for any $m< \omega$ and any $(s_0, s_1)\in S_0(m)\times S_1(m)$, there are $n> m$ and a product coding node $(t_0, t_1)\in S_0(n)\times S_1(n)$ with $s_i\sqsubseteq t_i$ for each $i< 2$.

    Given a product Devlin tree $S$, let $\cdnd(S)\subseteq M$ denote the set of product coding nodes. Write $\str^*(S)\subseteq \bM^*$ for the substructure induced on $\cdnd(S)$, with reduct $\str(S)\subseteq \bM$; if $\str(S)\cong \bK$ and $\phi\in \emb(\bK, \bM)$ is such that $\im(\phi) = \cdnd(S)$, then $(\bM^*{\cdot}\phi)/\bK$ will be a finitary, recurrent BRS. 
\end{exa}

\begin{exa}
    \label{Exa:Rado}
	The \emph{Rado graph} $\bK$ is the \fr limit of the class $\cK$ of finite graphs. Erd\H{o}s, Hajnal, and P\'osa \cite{EHP_1975} showed that $\rm{BRD}(e, \bK)\geq 4$, where $e$ denotes the edge. Pouzet and Sauer \cite{PS_1996} proved that $\rm{BRD}(e, \bK) = 4$. Using Milliken's theorem, Sauer \cite{Sauer_2006} proved that $\bK$ has finite BRDs, and shortly thereafter, Laflamme, Sauer, and Vuksanovic \cite{Laflamme_S_V} characterise the exact BRDs, in the process showing that $\bK$ admits a finitary, recurrent BRS.

    We view graphs as symmetric, irreflexive binary structures using only the binary relation $R_1$. Write $\bM = \bT_2^*|_{\{R_1\}}$; then $\bM$ is a graph, and it is straightforward to show that $\bM$ and $\bK$ are bi-embeddable. Thus upon building $\phi\in \emb(\bK, \bM)$ such that $(\bT_2^*{\cdot}\phi)/\bK$ is $4$-unavoidable, Theorem~\ref{Thm:Implies_Finite_Recurrent_BRS} tells us that this is a finitary, recurrent BRS. Similar to Example~\ref{Exa:Rationals}, such $\phi$ can be described by objects, again closely related to Joyce trees, which in this survey we call \emph{LSV trees} after the authors of \cite{Laflamme_S_V}. The definition is the exact same as for Devlin trees except that on coding levels, if $t\in \bS(m)$ is the coding node, $s\in \bS(m)\setminus \{t\}$, and $\is_\bS(s) = \{s'\}$, we allow the possibility that $R_1(s, s')$ holds. We also define \emph{LSV coding trees} similarly to the last example, and the \emph{structure coded by} $\bS$, denoted $\str(\bS)$, is the graph $\bM|_{\cdnd(\bS)}$. It is routine to construct LSV trees which code any countable graph. One of the main results of \cite{Laflamme_S_V} is that any two LSV-trees coding the Rado graph are bi-embeddable; while similar in spirit to the corresponding result for Devlin trees coding the rational order, the proof for LSV trees is harder (in fact, provably so \cite{ACDMP}). Hence if $\phi\in \emb(\bK, \bM)$ has the property that $\ol{\im(\phi)}^{\bT_2^*}$ is an LSV tree whose coding nodes are exactly $\im(\phi)$, then $(\bT_2^*{\cdot}\phi)/\bK$ is recurrent, and thus a big Ramsey structure. Note that formally, the result that all LSV coding trees which code the Rado graph are bi-embeddable is stronger than the assertion that $(\bT_2^*{\cdot}\phi)/\bK$ is recurrent; not only are $\ol{\im(\phi\circ \eta)}^{\bT_2^*}$ and $\ol{\im(\phi)}^{\bT_2^*}$ bi-embeddable whenver $\eta\in \emb(\bK)$, but \emph{any} LSV tree coding Rado has the form $\ol{\im(\phi\circ \eta)}^{\bT_2^*}$ for some $\eta\in \emb(\bK)$. Here we give a new proof of the bi-embeddability result using Theorem~\ref{Thm:Milliken_Envelopes}.

    Fix two LSV coding trees $\bS$ and $\bT$ such that $\str(\bT)$ is a Rado graph. We show that $\emb(\bS, \bT)\neq \emptyset$; note that this happens iff $\emb(\bS|_{\cdnd(\bS)}, \bT)\neq \emptyset$. We can suppose for every $m< \rm{ht}(\bS)$ that $\bS(m)\subseteq 2^m$. Let $\theta\in \emb(\bM, \str(\bT))$. Find $\eta\in \emb(\bT_2^*)$ such that for each $\bA^*\in \age(\bT_2^*)$ with $|A|\leq 4$, the map on $\emb(\bA^*, \bT_2^*)$ given by $f\to \bT{\cdot}(\theta\circ \eta\circ f)$ is monochromatic. To ease notation, we simply replace $\theta$ with $\theta\circ \eta$. So to emphasize, we have found $\theta\in \emb(\bM, \str(\bT))$ satisfying
    \begin{center}
        $(**)$ - For each $\bA^*\in \age(\bT_2^*)$ with $|A|\leq 4$, the map on $\emb(\bA^*, \bT_2^*)$ given by $f\to \bT{\cdot}(\theta\circ f)$ is monochromatic.
    \end{center}
    We will show that for any $\theta$ satisfying $(**)$, either $\theta|_{\cdnd(\bS)}\in \emb(\bS|_{\cdnd(\bS)}, \bT)$ or $\theta|_{\cdnd(\bS)}$ is almost an embedding, but flips $\lex$.
    In the items below, say that a finite subset $F\subseteq 2^{<\omega}$ is \emph{in LSV position} if $\ol{F}^{\bT_2^*}$ is an LSV-tree whose coding nodes are exactly $F$. 

    \begin{itemize}
        \item 
        For $x, y\in 2^{<\omega}$ with $x<_{ht} y$, we have $\theta(x)<_{ht}^{\bT} \theta(y)$. Write $\bA^*\subseteq \bT_2^*$ for the substructure induced on $\{x, y\}$. The set $\{f\in \emb(\bA^*, \bT_2^*): f(x) = x)\}$ is infinite. Thus the set $\{\theta(f(y)): f\in \emb(\bA^*, \bT_2^*)\}$ is infinite. As some of these must satisfy $\theta(x)<_{ht}^\bT \theta(f(y))$, assumption $(**)$ tells us that they all do.
        \item 
        For any $x, y\in 2^{<\omega}$, we have $(\theta(x)\wedge^\bT \theta(y)) \leq_{ht}^\bT \theta(x\wedge y)$. If $x\wedge y\in \{x, y\}$, this is clear. If $x\wedge y\not\in \{x, y\}$, then in the graph $\bM$, exactly one of $x$ or $y$ is adjacent to $x\wedge y$. Thus in $\str(\bT)$, exactly one of $\theta(x)$ or $\theta(y)$ is adjacent to $\theta(x\wedge y)$. As $\theta(x\wedge y) <_{ht}^\bT \{\theta(x), \theta(y)\}$ by the first item, we must have $\theta(x)|^\bT_{\rm{ht}_\bT(\theta(x\wedge y))} \neq \theta(y)|^\bT_{\rm{ht}_\bT(\theta(x\wedge y))}$,  implying that $(\theta(x)\wedge^\bT \theta(y)) <_{ht}^\bT \theta(x\wedge y)$.

        \item 
        For $w, x, y, z\in 2^{<\omega}$ with $(w\wedge x) <_{ht} (y\wedge z)$, then also $(\theta(w) \wedge^\bT\theta(x)) <_{ht}^\bT (\theta(y)\wedge^\bT \theta(z))$. We first consider the case $w = x$. Suppose $\rm{ht}(x) = m-1$. Find $\zeta\in \emb(\bT_{2, m}^*, \bT_{2, m}^*)$ so that for each $v\in 2^{<\omega}$ with $x<_{ht} v$, we have that $\theta\circ \zeta(v)|_{\rm{ht}_\bT(\theta(x))}$ depends only on $v|_{m-1}$. Hence what we want is true if $y, z\in \im(\zeta)$, but by assumption $(**)$, it must be true generally. When $w\neq x$, we have $(\theta(w)\wedge^\bT \theta(x)) \leq_{ht}^\bT \theta(w\wedge x) <_{ht}^\bT (\theta(y)\wedge^\bT \theta(z))$, the first inequality by the second item and the second using $w\wedge x$ in place of $x$ in the prior reasoning. 
        \item 
        From the first and third items, it follows that for any set $Q\subseteq 2^{<\omega}$ in LSV position, $\theta|_{Q}\in \emb(\bS^-|_Q, \bT^-)$, where $\bS^-$, $\bT^-$ are the $\lex$-forgetting reducts of $\bS$ and $\bT$.
        \item 
        On any set $Q\subseteq 2^{<\omega}$ in LSV position, $\theta$ either preserves $\lexeq$ or flips it. Suppose $\bA^*\in \age(\bT_2^*)$ has size $2$ and describes a pair of vertices in LSV position. By $(**)$ and the fourth bullet it is true that either $\bT{\cdot}(\theta\circ f) = \bA^*$ for every $f\in \emb(\bA^*, \bT_2^*)$ or $\bT{\cdot}(\theta\circ f)$ is the $\lexeq$-flip of $\bA^*$ for every such $f$. Supposing the former, we will show that $\theta$ preserves $\lexeq$ on $Q$; the proof that $\theta$ flips $\lexeq$ in the other case is similar. Suppose $\bB^*\in \age(\bT_2^*)$ describes another pair in LSV position. We can find $w\lex x\lex y \lex z\in 2^{<\omega}$ in LSV position such that $\bT_2^*|_{\{w, y\}}\cong \bA^*$, $\bT_2^*|_{\{x, z\}}\cong \bB^*$, and $w\wedge y = x\wedge z <_{ht}^\bT \{w\wedge x, y\wedge z\}$. By the fourth bullet, we have $\theta(w\wedge y) = \theta(x\wedge z) <_{ht}^\bT \{\theta(w)\wedge^\bT \theta(x), \theta(y)\wedge^\bT \theta(z)\}$, and our assumption on $\bA^*$ gives $\theta(w)\lex^\bT \theta(y)$. It follows that we must have $\theta(x)\lex^\bT \theta(z)$; appealing to $(**)$, this happens for every pair inducing a copy of $\bB^*$ in $2^{<\omega}$, and $\bB^*$ was arbitrary.
    \end{itemize}
    The above bullets combine to imply that $\theta|_{\cdnd(\bS)}\in \emb(\bS|_{\cdnd(\bS)}, \bT)$ in the case that $\theta$ was $\lexeq$-preserving. In the case that it was $\lexeq$-reversing, we can obtain the desired result by explicitly building an LSV-tree $\bU$ that embeds every LSV-tree (including $\bS$ and its $\lexeq$-flip), then running the above argument.  \qed
\end{exa}

\begin{exa}
    \label{Exa:Uncountable}
    In this example, we modify Definition~\ref{Def:tree} by weakening item (1) to allow $\rm{Pred}_\bT(t)$ to be well-ordered and possibly infinite. Shelah \cite{Shelah_288} proves a version of Milliken's theorem for an expansion of the uncountable tree $\bT_{\kappa} := \la 2^{<\kappa}, \preceq, \wedge, \leq_{ht}, \lex, R_1\ra$, where $\kappa$ is a cardinal which is measurable in the forcing extension obtained by adding $\beth_{\kappa+\omega}$-many Cohen reals. The needed expansion of $\bT_\kappa$, denoted $\bT_\kappa^\prec$, is given by adding a binary relation $\prec$ which well-orders every level; call structures of this form \emph{level-ordered trees}. However, the version of Milliken's theorem proven for $\bT_\kappa^{\prec}$ is weaker than showing that finite level-ordered trees are big Ramsey objects. Instead, given a finite level-ordered tree $\bA^{\!\prec}$ and a coloring $\gamma\colon \emb(\bA^{\!\prec}, \bT_\kappa^\prec)\to \sigma$ (where in fact any cardinal $\sigma < \kappa$ will work), one shows that there is $\eta\in \emb(\bT_\kappa, \bT_\kappa)$ (\emph{without} $\prec$) so that $\gamma$ is constant on $\eta[\emb(\bA^{\!\prec}, \bT_\kappa^\prec{\cdot}\eta)]$ (note the second appearance of $\eta$). Using a strengthening of this, D\v{z}amonja, Larson, and Mitchell \cite{DLM_Rational, DLM_Rado} construct finitary big Ramsey structures for the $\kappa$-saturated analogs of the rational order and the Rado graph. These simply look like uncountable Devlin trees and LSV trees, respectively, equipped with well-orders on every level satisfying certain properties. Whether or not any of these big Ramsey structures is recurrent is open.
\end{exa}

\begin{rem}
    A general limitation of the big Ramsey degree proofs which use Milliken's theorem as given in Theorem~\ref{Thm:Milliken} is
	the fact that a regularly branching tree can only represent structures in finite languages containing binary relations.  Two kinds of structures can be
	represented in the tree:
	\begin{enumerate}
		\item Linear orders (represented the lexicogrpahic order of the tree) as in Example~\ref{Exa:Rationals}.
		\item Unrestricted structures in binary language (represnted using the passing numbers) as in Example~\ref{Exa:Rado}. 
	\end{enumerate}
    Using the product form of Milliken's theorem, as in Example~\ref{exa:2orders} these constructions can be combined giving upper bounds on big Ramsey degrees for structures with multiple linear orders
    and multiple types of binary edges in free superposition. This can be used, for example, to  represent directed graphs (where the orientation
    of the edge is encoded by a pair of symmetric binary relations and an edge).  These structures are sometimes referred to
	as an ``unrestricted'' structures (or ``simple'' in~\cite{dobrinen2016rainbow}).

    To obtain upper bounds for BRDs of structures with relations of higher arity, it is possible to use a product form
    of Milliken's theorem on rapidly branching trees, which is beyond the scope of this survey. For details see~\cite{Stevo_book} or the paper~\cite{BCHKV_3Unif}, which gives upper
    bound on the BRDs in the class of 3-uniform hypergraphs and~\cite{BCdRHKK} generalizing this construction to unrestricted $\omega$-categorical
    structures in relational languages which are not necessarily finite, but are required to have finitely many relations of
    every arity.

    It is possible to construct a persistent coloring showing that the Rado graph with countably many types of edges
	does not have finite BRDs and thus, in a certain sense, these constructions are tight.
\end{rem}

\section{Coding tree techniques}
\label{Sec:Coding_Trees}

In a major technical leap forward, Dobrinen in \cite{Dobrinen_3Free, Dobrinen_Henson} proved that for every $k\geq 3$, the \fr class of finite $k$-clique-free graphs has finite big Ramsey degrees. For these classes, the corresponding \fr limits are much more difficult to code as trees in a straight-forward way. To get around this, Dobrinen introduces the concept of trees enriched with a designated set of coding nodes, which here we view as the unary predicate $\cdnd$.

Throughout this section, $\bK$ denotes an infinite structure with the property that $\cdnd\in \cL_\bK$ and $\cdnd(\bK) = K$. Abstractly, a ``proof of finite BRDs and/or existence of BRS using coding trees'' can be defined as an application of one of Theorems~\ref{Thm:Weak_Biembeddability} or \ref{Thm:Implies_Finite_Recurrent_BRS} where $\psi|_{\cdnd(\bM)}\colon \bM|_{\cdnd(\bM)}\to \bK$ is an isomorphism. We have already seen examples of weak bi-embeddings of this form; for instance, if $\bK$ is the rational order or the Rado graph, then we can let $\bM^*$ be a Devlin or LSV coding tree, respectively, and $\bM$ can be the reduct of $\bM^*$ eliminating all of the tree structure. So $\bM$ is just a copy of $\bK$ along with several extra isolated vertices that will become the other nodes of the tree, and $\psi$ can be defined arbitrarily on these extra nodes.

A common extra feature of coding tree arguments in the literature is that one typically attempts to prove Ramsey theorems about $\bM^*$ directly, coding nodes and all. This typically requires the use of techniques from set-theoretic forcing, an approach pioneered by Dobrinen \cite{Dobrinen_3Free} and inspired by a forcing proof of the Halpern-L\"auchli theorem\footnote{The Halpern-L\"auchli theorem is the inductive step in the proof of Milliken's theorem.} \cite{Halpern_Lauchli} given by Harrington which first appears in print in \cite{Todorcevic_Farah}, see also \cite{Dobrinen_forcing}. As for which $\bM^*$ to work with, the approaches vary. The two extreme situations are on the one hand, that $\bM^*$ is built in some canonical fashion, but finding $\phi$ as in Theorem~\ref{Thm:Implies_Finite_Recurrent_BRS} takes work, and on the other, that $\bM^*$ is rather difficult to build, but we can apply Theorem~\ref{Thm:Implies_Finite_Recurrent_BRS} taking $\phi$ to be onto $\bM|_{\cdnd(\bM)}$. Of course, there are examples where the construction of $\bM^*$ lies somewhere in between, for instance in \cite{Dobrinen_3Free, Dobrinen_Henson}. 

In what remains of this section, we describe these two extreme approaches and indicate several recent examples of them in the literature. First, we describe the most common ``canonical'' choice of $\bM^*$, the \emph{coding tree of $1$-types} with respect to a given \fr class.

\begin{defin}
    \label{Def:Tree_of_1types}
    Fix a \fr class $\cK$ with limit $\bK$ and with $\cL_\bK$ finite. We can ensure that $\cL_\bK$ doesn't contain any level tree relation symbols (Definition~\ref{Def:tree}). Let $\bA\leq \bK$ be enumerated. The \emph{$\cK$-coding tree} of $\bA$ \cite{BCDHKVZ}, or equivalently the \emph{tree of $\cK$-$1$-types} \cite{CDP_SDAP1, CDP_SDAP2}, has been denoted by either $\ct^\bA$ or $\bbS(\bA)$, respectively. Here, we adopt the notation $\bbT^\bA_\cK$. The underlying set of this tree is $\rm{tp}^\bA_\cK)$, the \emph{set of $1$-types over initial segments of $\bA$}.  Fix some symbol $\sft\not\in \omega$, the \emph{type vertex}. Members of $\rm{tp}^\bA_\cK$ are structures $\bB\in \cK$ with $B = n\cup \{\sft\}$ and $\bB|_{n} = \bA|_{n}$ for some $n< 1+|A|$. The tree order is simply that of substructure; note that $\bB\in \bbT^\bA_\cK(n)$ iff $B\setminus \{\sft\} = n$. In addition to the tree relations, we also equip $\bbT^\bA_\cK$ with the unary relation $\cdnd$, where if $\bB\in \rm{tp}^\bA_\cK$, then $\bB\in \cdnd(\bbT^\bA_\cK)$ iff for some $n< \omega$ we have $\bB\cong \bA|_{n+1}$ via the map $i\to i$ $(i< n)$ and $\sft\to n$. Notice that there is exactly one coding node per level. We then equip the coding nodes with relations from $\cL_\bK$ so that $(\bbT^\bA_\cK|_{\cdnd(\bbT^\bA_\cK)})|_{\cL_\bK}\cong \bA$ in the obvious way. 
    
    When $\cL_\bK$ is binary, then by re-encoding $\bK$, we can suppose $\cL_\bK\setminus \{\cdnd\} = \{U_j: j < k^\sfu\}\cup \{R_i: 0< i< k\}$ for some $k, k^\sfu< \omega$ with the $U_j$ unary and the $R_i$ binary, that $\neg R_i^\bK(a, a)$ holds for every $a\in K$, that $\{U_j(\bK): j< k^\sfu\}$ is a partition of $\bK$, and that the sets $\{R_i^\bK: 0< i< k\}$ are pairwise disjoint. In this case, one can encode members of $\rm{tp}^\bA_\cK$ as members of $k^\sfu\times k^{<\omega}$, where for a given $\bB\in \bbT^\bA_\cK(n)$, one simply encodes the unary relation on $\sft$ and the binary relations between the vertex $\sft$ and the vertices $0, \ldots, n-1$ in $\bB$. Form the tree 
    $$k^\sfu \times \bT_k := \la k^\sfu\times k^{<\omega}, \sqsubseteq, \wedge, \leq_{ht}, \lexeq, (U_j)_{j< k^\sfu}, (R_i)_{0< i< k}\ra$$ 
    by placing $k^\sfu$-many copies of $\bT_k$ side by side, equipping each $\bT_k$ with exactly one of the unary relations, and extending $\lexeq$ in the obvious way. Then the tree structure on $\bbT^\bA_\cK$ is induced from $k^\sfu\times \bT_k$, and the binary relations of the copy of $\bA$ induced on $\cdnd(\bbT^\bA_\cK)$ are given by the passing number relations $\{R_i: 0< i< k\}$.

    We also discuss another variant of the coding tree of $1$-types, called the \emph{unary-colored coding tree of 1-types} in \cite{CDP_SDAP1, CDP_SDAP2} and there denoted $\bbU$, or what was in \cite{Zucker_BR_Upper_Bound} called $\ct^\bA$. Here, we call this tree $\bbU^\bA_\cK$. We relax our assumptions to allow $\cL_\bK$ to possibly contain infinitely many unary symbols, but we demand that for every $a\in K$, there is exactly one unary $U\in \cL_\bK\setminus \{\cdnd\}$ with $a\in U(\bK)$. The tree $\bbU^\bA_\cK$ has underlying set $\rm{tp}^-(\bA, \cK)$, where members of $\rm{tp}^-(\bA, \cK)$ are $\cL_\bK$-structures $\bB$ with $B = n\cup\{\sft\}$, $\bB|_n = \bA|_n$, $\sft\not\in U(\bB)$ for any unary $U\in \cL_\bK\setminus \{\cdnd\}$, but there is some unary $U\in \cL_\bK$ such that upon equipping $\sft$ with unary $U$, we obtain a structure in $\cK$. In addition to the unary $\cdnd$ relation, we endow the coding nodes with unary relations from $\cL_\bK$ according to the unaries in $\bA$. When $\cL_\bK$ is binary and correctly encoded as above, we can encode members of $\rm{tp}^-(\bA, \cK)$ as members of $k^{<\omega}$, and the tree structure on $\bbU^\bA_\cK$ is induced from that of $\bT_k$. 
\end{defin}

\begin{exa}
    \label{Exa:BinaryFreeAmalg}
    Recall that a structure $\bA$ is called \emph{irreducible} if every $a\neq b\in \bA$ participates in some relation of $\bA$. Equivalently, this happens exactly when $\bA$ is not a free amalgam of two proper substructures. Given $\cL\subseteq \bbL$ and $\cF$ a set of finite, irreducible $\cL$-structures, we define $\rm{Forb}^\cL(\cF)$ as the class of finite $\cL$-structures $\bA$ for which $\bF\not\leq \bA$ for every $\bF\in \cF$.  Recall that a \fr class $\cK$ has \emph{free amalgamation} iff $\cK$ has the form $\rm{Forb}^\cL(\cF)$ for some $\cL\subseteq \bbL$ and set $\cF$ of finite, irreducible $\cL$-structures.
    
    Now suppose $\cK$ is a \fr binary free amalgamation class with enumerated limit $\bK$ and with $\cL_\bK$ finite. We assume that the enumeration of $\bK$ is \emph{left dense} (see \cite{Zucker_BR_Upper_Bound}), meaning that above any $v\in \rm{tp}^\bK_\cK\subseteq k^\sfu\times k^{<\omega}$, one can extend $v$ by adding a string of zeros to reach some member of $\cdnd(\bbT^\bK_\cK)$. To show that $(\bbT^\bK_\cK)^{*\rm{All}}$ satisfies a large fragment of IRT, we need to add more structure to the coding tree of $1$-types. Given an enumerated $\bA\leq \bK$, we form the \emph{aged coding tree of $1$-types}, which here we denote by $\bb{AT}^\bA_\cK$. This expansion adds new relations (not necessarily binary) that can hold on level tuples from $\bbT^\bA_\cK$. When $\bA$ is finite, the extra relations on $\bb{AT}^\bA_\cK$ can be viewed as describing what configurations of coding nodes could possibly appear in $\bbT^\bB_\cK$  above a given level tuple of nodes from $\bbT^\bA_\cK$, where $\bB\leq \bK$ is an enumerated structure with $\bB|_{A} = \bA$.
    
    The idea of adding extra structure to levels of trees is implicit in Dobrinen's work on the triangle-free and $k$-clique-free Henson graph \cite{Dobrinen_3Free, Dobrinen_Henson} via the concept of \emph{parallel 1s}. Indeed, if two vertices $x$ and $y$ in a triangle free graph are both adjacent to some vertex $z$, then $x$ and $y$ cannot be adjacent. Notationally, we remark that in \cite{Zucker_BR_Upper_Bound, BCDHKVZ, DZ}, instead of forming the expansion $\bb{AT}^\bA_\cK$ explicitly, these references refer to a submonoid of embeddings called \emph{aged embeddings}, written $\aemb(\ldots)$. Here, we instead expand the structure, so that $\aemb(\bbT^\bA_\cK, \bbT^\bB_\cK) = \emb(\bb{AT}^\bA_\cK, \bb{AT}^\bB_\cK)$.

    Crucially, the assumption that $\bK$ is left dense implies that $\bb{AT}^\bK_\cK/\bbT^\bK_\cK$ is embedding faithful. This expansion yields the following generalization of Milliken's theorem.
    \begin{theorem}[Theorem 3.5 of \cite{Zucker_BR_Upper_Bound}]
        \label{Thm:Aged_Milliken}
        $\rm{BRD}(\bb{AT}^\bA_\cK, \bb{AT}^\bK_\cK) = 1$ for every enumerated $\bA\in \cK$.
    \end{theorem}
    We note that the above discussion works similarly for $\bbU^\bA_\cK$, thus yielding the \emph{aged unary-colored coding tree of $1$-types} $\bb{AU}^\bA_\cK$; this is actually what is used in \cite{Zucker_BR_Upper_Bound}. Theorem~\ref{Thm:Aged_Milliken} gives us a large set of big Ramsey objects, and as in the previous section, one needs to check that this structure admits envelopes.
    \begin{theorem}
        \label{Thm:Aged_Envelopes}
        Letting $\cC = \{\bb{AT}_\cK^\bA: \bA\in \cK \text{ enumerated}\}$, we have that $\bb{AT}_\cK^\bK$ admits $\cC$-envelopes. Thus $(\bb{AT}_\cK^\bK)^{*\cC}$ satisfies IRT.
    \end{theorem}

    Unlike Theorem~\ref{Thm:Milliken_Envelopes}, $(\bb{AT}_\cK^\bK)^{*\cC}/\bb{AT}_\cK^\bK$ is not in general finitary. This might seem like an insurmountable obstacle. However, by adding extra assumptions to $\cK$, we can find a work-around. As we are assuming $\cK$ is a binary free amalgamation class, we have $\cK = \rm{Forb}^{\cL_\cK}(\cF)$ for some set $\cF$ of finite irreducible $\cL_\bK$-structures. We say that a free amalgamation class is \emph{finitely constrained} if we can take $\cF$ to be finite. From now on, assume $\cK$ is finitely constrained. Then, writing $\bM = \bbT^\bK_\cK|_{\cL_\bK}$, one can find $\phi\in \emb(\bK, \bM)$ so that $((\bb{AT}^\bK_\cK)^{*\cC}{\cdot}\phi)/\bK$ is finitary.
    
    In particular, \emph{there exists} some $\phi$ as above so that  $((\bb{AT}^\bK_\cK)^{*\cC}{\cdot}\phi)/\bK$ is $I$-finitary and $I$-unavoidable for some finite $I\subseteq \cK$; thus Theorem~\ref{Thm:Implies_Finite_Recurrent_BRS} implies that $\bK$ admits a finitary, recurrent BRS. However, in order to \emph{characterise} the finitary, recurrent BRSs of $\bK$, one must explicitly build such a $\phi$. This is the work undertaken in \cite{BCDHKVZ}. In particular, the \emph{diaries} defined and constructed in \cite{BCDHKVZ} are abstractions of the key features that the image of such a $\phi$ must have. It would be interesting to reprove the bi-embeddability result for diaries from \cite{BCDHKVZ} in the style of our proof from Example~\ref{Exa:Rado} that LSV trees are bi-embeddable.

    Another approach, taken for instance in \cite{CDP_SDAP1}, \cite{CDP_SDAP2}, \cite{Dobrinen_SDAP_Inf} and \cite{DZ}, is to postpone proving Ramsey theorems until after constructing $\phi$ as above. This is closer in spirit to the second extreme discussed before Definition~\ref{Def:Tree_of_1types}. Regarding finitely constrained binary free amalgamation classes, Dobrinen and Zucker build particularly nice diaries called \emph{strong diaries} and prove directly using forcing that these satisfy large fragments of IRT.

    This concludes the discussion of Example~\ref{Exa:BinaryFreeAmalg}.
\end{exa}

\begin{exa}
\label{Exa:CDP}
    In the papers \cite{CDP_SDAP1, CDP_SDAP2}, Coulson, Dobrinen, and Patel isolate, in decreasing order of strength, the amalgamation properties $\rm{SFAP}$, $\rm{SDAP}^+$, and $\rm{LSDAP}^+$ that a \fr class $\cK$ might enjoy. We refer there for the definitions, but the idea is that for binary $\rm{LSDAP}^+$ classes, it is straightforward to modify $\bbU^\bK_\cK$ into a recurrent big Ramsey structure for $\bK$, and one can prove IRT directly on this big Ramsey structure using forcing. This results in a characterisation of big Ramsey degrees and recurrent big Ramsey structures of a particularly simple nature. Fix a \fr class $\cK$ with $\rm{LSDAP}^+$ along with an enumerated \fr limit $\bK$. We assume $\cL_\bK$ is finite, that $\{U\in \cL_\bK\setminus \{\cdnd\}: U \text{ unary}\} = \{U_j: j< k^\sfu\}$ for some $k^\sfu < \omega$, and that for each $a\in K$, there is exactly one $j< k^\sfu$ with $a\in U_j(\bK)$. The authors show that there is a partition $\Pi$ of $k^\sfu$ such that, writing $\bM = \bbU^\bK_\cK|_{\cL_\bK}$, then if $\phi\in \emb(\bK, \bM)$ has the properties that 
    \begin{itemize}
        \item 
        $\im(\phi)$ is an antichain in $\bbU^\bK_\cK$,
        \item 
        for some $\ell< \omega$ with $\rm{ht}_{\bbU^\bK_\cK}(\phi(a)) > \ell$ for every $a\in K$, there are distinct $\{v_P: P\in \Pi\}\subseteq \bbU^\bK_\cK(\ell)$ such that $\phi(a)|^{\bbU^\bK_\cK}_\ell = v_P$ iff $a\in U_j(\bK)$ for some $j\in P$,
    \end{itemize}
    then $((\bbU^\bK_\cK){\cdot}\phi)/\bK$ is recurrent (and finitary). Furthermore, if $\cL_\bK$ is binary, then the authors prove directly using forcing that $(\bbU^\bK_\cK){\cdot}\phi$ satisfies IRT, thus yielding a BRS. When $\bK$ is not binary, the authors show that vertices are big Ramsey objects (i.e.\ that the \fr limits are \emph{indivisible}).
    
    Examples of binary $\rm{LSDAP}^+$ classes include the class of finite linear orders with a vertex partition into $k^\sfu$-many unaries, whose BRDs were characterised in earlier work of Laflamme--Nguyen Van Th\'e--Sauer \cite{L_NVT_S_2010}, the class of finite graphs with a vertex partition into $k^\sfu$-many unaries (in particular, when $k^\sfu = 1$, they obtain another proof of the bi-embeddability result discussed in Example~\ref{Exa:Rado}), and the class of finite linear orders equipped with a convex equivalence relation. For the partitioned orders, we have $\Pi = \{k^\sfu\}$, while for the partitioned graphs, we have $\Pi = \{\{j\}: j< k^\sfu\}$. The idea for this difference is that in graphs, we can use the edge relation with respect to some ``external'' vertex to separate the unaries, but in linear orders, this is not possible using an external vertex and the order relation. Thus for the partitioned orders, BRSs for the \fr limit are described by Devlin trees where the coding nodes get an additional unary label, whereas for the partitioned graphs, the BRSs for the \fr limit are described by LSV trees with $k^\sfu$-many labeled roots. For the linear orders with convex equivalence relations, the BRSs for the \fr limit are simply described by Devlin trees $\bS$ equipped with a $\lex$-convex equivalence relation $\sim$ on the coding nodes with the property that if $a\lex b\lex c \in \cdnd(\bS)$, $a\sim b$, and $b\not\sim c$, then $a\wedge c <_{ht} a\wedge b$, and if instead $a\not\sim b$ and $b\sim c$, then $a\wedge c <_{ht} b\wedge c$.

    For many more examples of \fr structures satisfying these amalgamation properties, we refer to the catalog at the end of \cite{CDP_SDAP2}.
\end{exa}

\section{Parameter space methods and generalizations}
\label{Sec:Parameter_Space}

As discussed at the end of Section~\ref{Sec:Milliken}, applications of Milliken's theorem are limited to big Ramsey degrees
of unrestricted structures. For certain restricted structures, it is possible to apply the Carlson--Simpson theorem~\cite{carlson1984}
in the place of Milliken's tree theorem. This technique was introduced in~\cite{Hubicka2020CS} and can be applied to structures
with unary and binary relations which are \emph{triangle constrained}, i.e.\ described by certain families of forbidden substructures with at most 3 vertices.

Fix $k>0$ and $\sigma\subseteq \{\vec{s}\in k^{{<}\omega}: |\vec{s}\,|>0\}$. Here we use tuple notation $\vec{s}$ to discuss members of $k^{<\omega}$ which we don't want to think of as members of $\bT_k$, and $|\vec{s}\,|$ denotes the length of $\vec{s}$ (which we don't want to think of as height in a tree). Indeed, we expand the structure $\bT_k$ introduced in Section~\ref{Sec:Milliken} into:
$$\bT^\sigma_k:= \la k^{{<}\omega}, \sqsubseteq, \wedge, \leq_{ht}, \lexeq, (R_i)_{0< i< k}, (R_{\vec{s}})_{\vec{s}\in \sigma}\ra.$$
where the relations $\sqsubseteq, \wedge, \leq_{ht}, \lexeq, (R_i)_{0< i< k}$ are as in Section~\ref{Sec:Milliken}.
For each $\vec{s}\in \sigma$, the relation $R_{\vec{s}}$ has arity $|\vec{s}\,|$ and we put $(t_0,t_1,\ldots,t_{|\vec{s}|-1})\in R_{\vec{s}}$ if and only if 
there exists $\ell<\min_{i<|\vec{s}\,|}\rm{ht}(t_i)$ such that $t_i(\ell)=\vec{s}(i)$ for every $i<|\vec{s}\,|$.

We can again expand further by adding a unary predicate $U$ to nodes of level less than $m$, and we write $\bT^\sigma_{k, m}$ for this structure.
For $\sigma=\{\vec{s}\in k^{{<}\omega}: |\vec{s}\,|>0\}$ the following is a direct consequence of the Carlson--Simpson theorem~\cite{carlson1984}.
For other choices of $\sigma$ this follows from~\cite{Balko2023Sucessor}.
\begin{theorem}
	\label{Thm:CS}
	For any $k, m< \omega$, we have $$\{\bT^\sigma_{k, m}({<}n): m\leq n< \omega\}\subseteq \rm{BRO}(\bT^\sigma_{k, m}).$$ 
\end{theorem}
Once again, one also needs to verify that the structure admits envelopes (see~\cite{Hubicka2020CS} for verification).
\begin{theorem}
    \label{Thm:CS_Envelopes}
    Writing $\cC = \{\bT^\sigma_{k, m}({<}n): m\leq n< \omega\}$, we have that $\bT_{k, m}^\sigma$ admits $\cC$-envelopes. Thus $(\bT_{k, m}^\sigma)^{*\cC}$ satisfies IRT.
\end{theorem}
When the $\cC$ from Theorem~\ref{Thm:CS_Envelopes} is understood, we write $\bT_{k, m}^{\sigma *}$ for $(\bT_{k, m}^\sigma)^{*\cC}$.
\begin{exa}
	The \emph{triangle-free Henson graph} $\bK$ is the \fr limit of the class $\cK$ of all finite triange-free graphs. We put $\sigma=\{\la 11 \ra\}$.
	We will proceed similarly as in Example~\ref{Exa:Rado} and view graphs as symmetric, irreflexive binary structures. We however change the definition of the graph $\bM$: we have $M = 2^{{<}\omega}$, and a pair of vertices $u, v$ with $u<_{\rm{ht}} v$ forms an edge if and only if
	$(u, v)\in R^{\bT_2^{\sigma*}}_1$ and $(u, v)\notin R^{\bT_2^{\sigma*}}_{\la 11\ra }$. There are no edges between vertices of the same level. While technically this graph is not a reduct of $\bT_2^{\sigma*}$, it is a reduct of an embedding faithful expansion of it.  

	We claim that $\bM$ is triangle-free. Suppose towards the contrary that the vertices $u$, $v$, and  $w$ forms a triangle.  Clearly we can assume $u <_{\rm{ht}} v<_{\rm{ht}} w$.
	Because $(u,w)\in R^{\bT_2^{\sigma*}}_1$ and $(u, v)\in R^{\bT_2^{\sigma*}}_1$ we have $v(\dom(u))=w(\dom(u))=1$
	which implies that $(v, w)\in R^{\bT_2^{\sigma*}}_{\la 11 \ra}$ contradicting the fact that $v$ and $w$ forms an edge of $\bM$.

	Moreover $\bM$ is bi-embeddable with $\bK$.  To see that assume that the vertex set of $\bK$ is $\omega$ and assign every vertex $i\in \omega$ a sequence $\varphi(i)\in 2^i$ where $\varphi(i)(j)=1$ if and only if $i$ and $j$ forms
	an edge of $\bK$.  It is easy to check that this is an embedding $\varphi:\bK\to \bM$ since for every $i<j\in \omega$ it holds that $(\varphi(i),\varphi(j))\notin R^{\bT_2^{\sigma*}}_{\la 11\ra }$ (which follows from the fact that $\bK$ is triangle-free).
	Then $\bT_2^{\sigma*}{\cdot}\phi$ is an expansion of $\bK$. 

	By Theorems~\ref{Thm:CS} and~\ref{Thm:CS_Envelopes} we obtain that the big Ramsey degrees of $\bK$ are finite.  The precise characterisation of big Ramsey structure was independently obtained by Balko, Chodounsk{\' y}, Hubi{\v c}ka, Kone{\v c}n{\' y}, Vena and Zucker and by Dobrinen~\cite{dobrinen2020ramsey}.
	Self-contained presentations of this characterisation appear in both~\cite[Example 3.4.4]{BCDHKVZ} and \cite[Definition 9.5]{BCDHKVZ_poset}. See also Example~\ref{Exa:BinaryFreeAmalg}.
\end{exa}
\begin{exa}
	\label{exa:poset}
	The \emph{generic partial order} $\bK$ is the \fr limit of the class $\cK$ of all finite partial orders. Notice that the tree of types
	of the generic partial order is ternary, since for every pair of vertices $u,v$ we have either $u<v$, $u>v$ or $u\perp v$.
	We associate $0,1,2$ with letters $\mathrm L,\mathrm X,\mathrm R$ respectively and imagine the partial order pictured in a way
	so all inequalities go from left to right:  $L$ denotes ``left'', $X$ is used for uncomparable and $R$ denotes ``right''.
	We will use the natural order $L<X<R$.

	Finiteness of big Ramsey degrees of the $\bK$ was shown by Hubi\v cka~\cite{Hubicka2020CS} by an application of the Carlson--Simpson theorem; recurrent big Ramsey structures for $\bK$ were characterised by Balko, Chodounsk{\' y}, Dobrinen, Hubi{\v c}ka, Kone{\v c}n{\' y}, Vena and Zucker~\cite{BCDHKVZ_poset} using a refinement of the Calrlson-Simpson theorem which motivated the introduction of parameter $\sigma$ in this survey.
	We follow the presentation of \cite{BCDHKVZ_poset} and
	define the following partial order on $3^{{<}\omega}=\{\mathrm L,\mathrm X,\mathrm R\}^{{<}\omega}$.

	\begin{defin}
	For $x,y\in 3^{{<}\omega}$, we set $x\prec y$ if and only if there exists $i$ such that:
	\begin{enumerate}
		\item $0\leq i<\min(\dom(x),\dom(y))$,
		\item $\la x(i)y(i)\ra=\la\mathrm L\mathrm R\ra$, and
		\item $x(j)< y(j)$ for every $0\leq j< i$.
	\end{enumerate}
	\end{defin}
	It is easy to verify (see \cite{Hubicka2020CS}) that $\bM=\la 3^{{<}\omega},{\prec}\ra$ is a partial oder and $\la 3^{{<}\omega},{\lex}\ra $ is its linear extension.
	Moreover $\bM$ is bi-embeddable with $\bK$.  To see this,  assume that the vertex set of $\bK$ is $\omega$ and assign to every vertex $i\in \omega$ a sequence $\varphi(i)\in 3^i$ where $\varphi(i)(j)=\mathrm L$ if and $i<^\bK j$, $\varphi(i)(j)=\mathrm R$ if and $i>^\bK j$ and $\varphi(i)(j)=\mathrm X$ otherwise.

	We put $$\sigma=\{\la LR \ra\}\cup \{\la ab \ra : a,b\in \{\mathrm L,\mathrm X,\mathrm R\}\hbox{ and } a>b\}.$$ 
	Notice that, by the choice of $\sigma$, any embedding $\bT^\sigma_{3,m}\to \bT^\sigma_{3,m}$ (for any $m$) is also an embedding $\bM\to \bM$.
	By Theorems~\ref{Thm:CS} and~\ref{Thm:CS_Envelopes} we obtain that big Ramsey degrees of $\bK$ are finite. For the precise characterisation of recurrent big Ramsey structures for $\bK$, see \cite[Definition 1.5]{BCDHKVZ_poset}. We remark that shape-preserving functions used in \cite{BCDHKVZ_poset} are coarser than embeddings of $\bT^\sigma_{3,m}\to \bT^\sigma_{3,m}$.
\end{exa}
	Balko, Chodounsk\'y, Hubi\v cka, Kone\v cn\' y, Ne\v set\v ril, and Vena~\cite{balko2021big} give simple, yet quite general structural condition on when methods based on the Carlson--Simpson theorem can be used to obtain finiteness of BRDs.

	Given structures $\bA$ and $\bB$, a \emph{homomorphism} from $\bA$ to $\bB$ is a map $\phi\colon A\to B$ so that images of related tuples remain related. As in~\cite{hubivcka2019all} we call a homomorphism $\varphi\colon \bA \to \bB$ a
	\emph{homomorphism-embedding} if $\varphi$ restricted to any irreducible substructure of $\bA$ is
	an embedding. The homomorphism-embedding $\varphi$ is called a \emph{strong completion}
	of $\bA$ to $\bB$ provided that $\bB$ is irreducible and $\varphi$ is injective.

	\begin{theorem}
 \label{Thm:Cycles}
	Let $\bK$ be a countably-infinite irreducible structure with $\cL_K$ finite and containing only unary and binary relation symbols. Assume that
	every countable structure $\bA$ has a strong completion to $\bK$
	provided that every induced cycle in $\bA$ (seen as a substructure) has a strong completion to $\bK$ and every irreducible substructure
	of $\bA$ of size at most 2 embeds into $\bK$. Then $\bA$ has finite big Ramsey degrees.
	\end{theorem}
	This result goes along the lines of the structural condition given for small Ramsey degrees in~\cite{balko2021big} and implies finiteness of BRDs for \fr limits of metric spaces with distances in $\{0,1,\ldots \delta\}$ for
	every finite diameter $\delta\in \omega$. More generally, Theorem~\ref{Thm:Cycles} can also be applied to the metric
	spaces associated to metrically homogeneous graphs of finite diameter
	with no Henson constraints as defined in Cherlin's catalogue~\cite{Cherlin2013}.
	The precise characterisations of the big Ramsey structures still needs to be given, but we believe
	that they can be obtained analogously as in~\cite{BCDHKVZ_poset}.
\begin{rem}
	While applications of Carlson--Simpson theorem are limited to triangle
	constrained structures, generalization of this method introducing a new
	Ramsey-type theorem on trees is given in~\cite{Balko2023Sucessor}.
	This result can be used to obtain all BRD upper bounds discussed in this survey as well as new ones, see announcement~\cite{typeamalg}.
\end{rem}

\section{Expansions by unary relations and functions}
\label{Sec:Functions}
In this section we consider structures in a language with both relation and unary function symbols.
We present a straightforward technique, based on~\cite{hubivcka2019all}, for extending structures with known big Ramsey bounds by incorporating additional unary relations and functions. While initially motivated by the study of 2-orientable graphs \cite{evans2019automorphism}, the method is described here in a more general form.

All examples of structures with finite big Ramsey degrees discussed so far use finite languages only.  This is not a coincidence, since there are persistent colorings showing, for example, that a random graph with $\omega$-many types of edges does not have finite BRDs \cite{HTZ_omega_edge}. Quite surprisingly however, structures with infinitely many unary relations may have bounded big Ramsey degrees as shown
in~\cite{BCdRHKK}. The construction of this section indicates, among other things, how to transfer known big Ramsey results to expansions with countably many new unaries.

\textbf{In this section, we fix} a language $\bbL'$ extending $\bbL$ by unary function symbols, as many as will ever appear in this section. We note that in this work, similar to \cite{hubivcka2019all}, we allow partial functions, so an embedding between $\bbL'$-structures $\bA$ and $\bB$ is an injection $g\colon A\to B$ which is an embedding from $\bA|_\bbL$ to $\bB|_\bbL$, and for every function symbol $F\in \bbL'$ and $a\in A$, we have $a\in \dom(F^\bA)\Leftrightarrow g(a)\in \dom(F^\bB)$, and when this happens, we have $g(F^\bA(a)) = F^\bB(g(a))$. Similar to our previous conventions, in this section, a \emph{structure} means an $\bbL'$-structure, and given a structure $\bK$, we set $\cL_\bK = \cL_{\bK|_\bbL}\cup \{F\in \bbL'\setminus \bbL: \dom(F^\bK)\neq \emptyset\}$. We remark that partial unary functions can be used to encode unary predicates by adding a unary function symbol $F$ whose domain is the desired unary predicate, setting $F(a) = a$ for each $a\in \dom(F)$.

Given a structure $\bA$, now that there are function symbols, it is not necessarily true that there exists a substructure with any given domain $B\subseteq A$. Given a set $S$ and a structure $\bA$, we say that $\bA$ is \emph{$S$-generated} if $S\subseteq A$ and there is no proper substructure of $\bA$ containing $S$.
We call $\bA$ \emph{locally finite} if every finite $S\subseteq A$ generates a finite substructure. Substructures generated by singletons, especially the singleton $\{0\} = 1$, will play an important role. Given $v\in \bA$, we denote by $\mathrm{Cl}_\bA(v)$ the substructure generated by $\{v\}$. Given a  $1$-generated structure $\bB$, we say that $\bB$ \emph{describes} $(\mathrm{Cl}_\bA(v), v)$, or just that $\bB$ \emph{describes} $\mathrm{Cl}_\bA(v)$ when $v$ is understood, iff there exists an embedding $e:\bB\to\bA$ such that $e(0)=v$ (such $e$, if it exists, is unique). 

\begin{defin}
	\label{Def:Unary}
	\begin{enumerate}
		\item A set $\cC$ of finite $1$-generated structures is \emph{description-closed} if whenever $\bA\in \cC$ and $v\in \bA$, there is $\bB\in \cC$ that describes $\mathrm{Cl}_\bA(v)$. We say $\cC$ is \emph{irredundant} if whenever $\bA,\bA'\in \cC$ and $f:\bA\to\bA'$ is an isomorphism satisfying $f(0)=0$, then $\bA=\bA'$. A \emph{set of allowed vertex closures} is a countable, description-closed, irredundant set of $1$-generated structures. 
  
	\item Given a structure $\bB$ and a relation symbol $R\in \bbL$ of arity $1\leq k< \omega$, we say that $\bB$ is \emph{$R$-generated} if $\bB$ is $k$-generated and $(0,...,k-1)\in R^\bB$. We call an $\bbL'$-structure \emph{simple} if it is $R$-generated for some $R\in \bbL$. 
  
	\item Let $\bK$ be an $\bbL$-structure, $\cC$ a set of allowed vertex closures, and $\cD$ a set of simple substructures. 
	      Denote by $\bK^{\cC,\cD}$ the following $\bbL'$-structure:
	\begin{enumerate}
		\item $K^{\cC,\cD}=\{(\bA,e): \bA\in \cC\hbox{ and $e$ is an embedding } \bA|_\bbL\to \bK\}$. 
		\item For every relational symbol $R\in \bbL$ of arity $k$  we put 
        $$((\bA_0,e_0),(\bA_1,e_1),\ldots, (\bA_{k-1},e_{k-1})) \in R^{\bK^{\cC,\cD}}$$ 
			if and only if there are an $R$-generated $\bB\in \cD$, an embedding $f\colon \bB|_\bbL\to \bK$, and embeddings $g_i\colon \bA_i\to \bB$ with $g_i(0) = i$ and $e_i = f\circ g_i$ for every $i<k$.

			\item For every function symbol $F\in \bbL'$ we put $(\bA_0, e_0)\in \dom(F^{\bK^{\cC, \cD}})$ iff $0\in \dom(F^{\bA_0})$, in which case we have $F^{\bK^{\cC, \cD}}((\bA_0,e_0))=(\bA_1,e_1)$ iff  $\bA_1$ describes $\rm{Cl}_{\bA_0}(F^{\bA_0}(0))$, and letting $e\colon \bA_1\to \bA_0$ be the unique embedding with $e(0) = F^{\bA_0}(0)$, we have $e_0\circ e = e_1$.
	\end{enumerate}
	Given $(\bA,e)\in K^{\cC,\cD}$ we put its \emph{projection} denoted by $\pi(\bA,e)$ to be $e(0)$.

	\item Given $\cC$ a set of allowed vertex-closures, we denote by $\bK^{\cC}$ the $\bbL'$-structure $\bK^{\cC,\cD}$ for $\cD$ being (up to isomorphic members) the inclusion maximal set of allowed simple substructures. 
	\end{enumerate}
\end{defin}

The purpose of $\cC$ and $\cD$ is to construct new structures $\bK^{\cC,\cD}$ from old in a way Ramseyness is transfered, too.
Let us discuss three simple examples showing possibilities of this construction.
\begin{exa}[Two unary functions]
	\label{exa:orientations}
	Let $\bK$ be a countable structure with $\cL_\bK = \emptyset$.  Let $\cL' = \{F_0, F_1\}\subseteq \bbL' \setminus \bbL$ (i.e. $F_0$ and $F_1$ are unary function symbols).
	Let $\cC$ be an inclusion maximal irredundant set of finite $1$-generated $\cL'$-structures. 
	Then every countable locally finite $\cL'$-structure has an embedding to $\bK^{\cC}$.
\end{exa}
\begin{exa}[Matching on the top of order of rationals]
	\label{exa:permutation}
	Let $\bK$ be the rational linear order as discussed in Example~\ref{Exa:Rationals}, using the binary relation symbol $\lex\in \bbL$, and let $F\in \bbL'\setminus \bbL$. Let $\cC$ consist of all structures $\bA$ with domain $\{0,1\}$, $\cL_\bA = \{\lex, F\}:= \cL'$, with $\lex$ a linear order, and with $F$ swapping $0$ and $1$. These are the following two structures:
	\begin{enumerate}
		\item $\bA_0$ with $0\prec_{lex}1$, $F^{\bA_0}(0)=1$, $F^{\bA_0}(1)=0$.
		\item $\bA_1$ with $1 \prec_{lex}0$, $F^{\bA_1}(0)=1$, $F^{\bA_1}(1)=0$.
	\end{enumerate}

	We can identify $K^{\cC}$ with $K^2\setminus \Delta_K$. Given $(a, b)$ and $(c, d)$ in $K^{\cC}$, we have $(a, b)\lex^{\bK^\cC} (c, d)$ iff $a\lex^\bK c$, and we have $F^{\bK^\cC}((a, b)) = (b, a)$. Note that $\lex^{\bK^\cC}$ is only a  partial order, since for instance $(a, b)\not\lex^{\bK^\cC} (a, c)$ and vice versa.
\end{exa}
\begin{exa}[Bipartite graph from Rado graph]
	\label{exa:rado}
	Consider the Rado graph $\bK$ discussed in Example~\ref{Exa:Rado}. Let $\cL' = \{R_1, F\}\subseteq \bbL'$, where $F\in \bbL'\setminus \bbL$. For this example, we will treat $F$ as a unary predicate, so we will only specify the domain of $F$. Let $\cC$ consist of the following two $\cL'$-structures:
	\begin{enumerate}
		\item $\bA$ with $A=\{0\}$, $R_1^\bA=\emptyset$, $\dom(F^\bA)=\emptyset$.
		\item $\bA'$ with $A'=\{0\}$, $R_1^{\bA'}=\emptyset$, $\dom(F^{\bA'})=\{0\}$.
	\end{enumerate}
	Additionaly let $\cD$ consist of two $\cL'$-structures:
	\begin{enumerate}
		\item $\bE$ with $E=\{0,1\}$, $R_1^\bE=\{(0,1),(1,0)\}$, $\dom(F^{\bA'})=\{0\}$.
		\item $\bE'$ with $E'=\{0,1\}$, $R_1^{\bE'}=\{(0,1),(1,0)\}$, $\dom(F^{\bA'})=\{1\}$.
	\end{enumerate}
	The structure $\bK^{\cC,\cD}$ is a bipartite graph created form the Rado graph $\bK$ by duplicating every vertex into two (where precisely one of the two copies is in the domain of $F$) and erasing all edges connecting vertices that are either both in $\dom(F)$ or both not in $\dom(F)$. Thus we obtain an universal bipartite graph.

	Notice that exactly the same construction can be used for $k$-partite graphs for every finite $k$
	as well as for $\omega$-partite graph.
	If $k$ is finite, we can consider either the graph with or without parts distinguished by unary relational
	symbols. For $\omega$-partite graph the parts needs to be determined by means of $\omega$ many unary relations
	(otherwise the graph would be isomorphic to the Rado graph).
\end{exa}
\begin{exa}[$\bS(2)$ from labelled order of rationals]
	\label{exa:S2}
\begin{figure}[h]
\setlength{\unitlength}{1mm}
\begin{picture}(20,30)(0,0)

\put(10,15){\circle*{1}}
\put(10,25){\circle*{1}}
\put(0,5){\circle*{1}}
\put(20,5){\circle*{1}}
\put(1,5){\vector(1,0){18}}
\put(9,14){\vector(-1,-1){8}}
\put(19,6){\vector(-1,1){8}}
\put(10,24){\vector(0,-1){8}}
\put(9,24){\vector(-1,-2){9}}
\put(11,24){\vector(1,-2){9}}

\end{picture}
	\label{fig:dD}
\caption{The tournament $\bD$}
\end{figure}
	The tournament $\bS(2)$ is defined as follows: let $C$ denote the unit circle in the
	complex plane. Define an oriented graph structure on $C$ by declaring that there is
	an directed edge from $x$ to $y$ iff $0 < \mathrm{arg}(y/x)<\pi$. Call $\bC$ the resulting oriented graph. The
	dense local order is then the substructure (oriented subgraph) $\bS(2)$ of $\bC$ whose vertices are those points
	of $\bC$ with rational argument. 
	Equivalently, $\bS(2)$ is a \fr limit of the class of all $\bD$-free tournaments where
	$\bD$ is a tournament depicted in Figure~\ref{fig:dD}.

	To describe the big Ramsey degrees of $\bS(2)$, Laflamme, Nguyen Van
	Th\'e, and Sauer~\cite{L_NVT_S_2010} employed two main approaches: they
	represented $\bS(2)$ using a free superposition of the order of
	rationals and a structure in a unary language~\cite{L_NVT_S_2010}, and
	they developed their own variant of the Milliken tree theorem where
	each level is colored by finitely many colours.  We give an alternate
	approach using Definition~\ref{Def:Unary} based on the same
	representation.

	Notice that fixing an arbitrary vertex $v\in \bS(2)$ divides the remaining vertices
	into two parts. The first part consists of all vertices $u$ such that $(u,v)$
	is an edge, while the second part consists of all vertices $u'$ such that $(v,u')$
	forms an edge.  Each part is isomorphic to the order of rationals. This
	motivates the following construction.
	
	Let $\bK$ be the rational linear order (as discussed in Example~\ref{Exa:Rationals}) where the order is represented using the binary relation symbol $\lex\in \bbL$.
	We will proceed in analogy to Example~\ref{exa:rado} and duplicate each vertex of $\bK$ into two distinguished copies.
	Let $\cL' = \{\lex, F\}\subseteq \bbL'$, where $F\in \bbL'\setminus \bbL$ is a unary function symbol. Let $\mathcal C$ consist of the following two structures:
	\begin{enumerate}
		\item $\bA$ with $A=\{0\}$, $R_1^\bA=\emptyset$, $\dom(F^\bA)=\emptyset$.
		\item $\bA'$ with $A'=\{0\}$, $R_1^{\bA'}=\emptyset$, $\dom(F^{\bA'})=\{0\}$, $F^{\bA'}(0)=0$.
	\end{enumerate}
	Now let $\mathcal L=\{R_1\}\subseteq \bbL$ be a language consisting of a single binary relation.
	Let $\bS$ be a structure with the same vertex set as $\bK^\cC$ where we put $(a,b)\in R_1^\bS$ if
	\begin{enumerate}
		\item $a\lex^{\bK^\cC} b$ and vertex closures of $a$ and $b$ in $\bK^\cC$ are isomorphic.
		\item $b\lex^{\bK^\cC} a$ and vertex closures of $a$ and $b$ in $\bK^\cC$ are not isomorphic.
		\item Neither $a\lex^{\bK^\cC} b$ or $b\lex^{\bK^\cC} a$ holds, $a\notin \dom F^{\bK^\cC}$ and $b\in \dom F^{\bK^\cC}$.
	\end{enumerate}
	As discussed in~\cite{L_NVT_S_2010}, we have $\bS\cong \bS(2)$, and $\bS$ is interpreted in $\bK^\cC$.
	Consequently finite big Ramsey degrees of $\bK^\cC$ also implies finite big Ramsey degrees of $\bS(2)$.

	We remark that this construction can be generalized to the structures $\bS(k)$ for any finite $k$, disucssed in \cite{DB,CDP_SDAP2},
	as well the order of rationals partitioned into $n$ dense or convex parts, for a given finite $n$ \cite{CDP_SDAP2}.
\end{exa}
\begin{exa}[Ultrametric spaces]
	\label{exa:ultrametric}
	Given $n\in \bbN$, an \emph{$n$-ultrametric space} is a metric space $\bA=(A,d)$ such that every $u,v,w\in A$ satisfies $d(u,w)\in n$ and $d(u,w)\leq \max \{d(u,v),d(v,w)\}$. 
	It is well known that the class of finite $n$-ultrametric spaces forms a \fr class, and we write $\bU_n$ for the \fr limit.

	Notice that for every $n$-ultrametric space $\bA$ and every distance $\ell<n$ the relation $\sim^\ell$ defined by putting $u\sim_\ell v$
	if and only if $d(u,v)\leq \ell$ is an equivalence relation. Every
	equivalence class of $\sim_\ell$ is called the \emph{ball of diameter
	$\ell$}.

	Put $\cL=\emptyset$ and let $\cL'$ contain a single unary function $F_1$. 
	For every $n$-ultrametric space $\bA$ we define an $\cL'$-structure $U(\bA)$ constructed as follows:
	\begin{enumerate}
		\item The vertex set is the set of all balls of $\bA$ (of every diameter $\ell<n$).
			Notice that the vertices of $\bA$ are balls of diameter 0.
		\item Given a ball $b$ of diameter $\ell<n-1$ we define $F_1(b)=c$ where $c$ is the unique ball of diameter $\ell+1$ containing $b$.
	\end{enumerate}

	Notice that:
	\begin{enumerate}
		\item $U$ is injective: it is possible to uniquely reconstruct the $n$-ultrametric space $\bA$ from $U(\bA)$.
		\item If $\bA$ and $\bB$ are $n$-ultrametric spacess and $\bA$ is a subspace of $\bB$, then there is an unique embedding $U(\bA)\to U(\bB)$ which is identity on balls of diameter 0. More generally, for every $e\in \emb(\bA,\bB)$, there is a corresponding $e'\in \emb (U(\bA),U(\bB))$.
		\item The closure of every ball $b\in U(\bA)$ of every diameter $\ell\in n$ is a sequence $b=b_\ell,b_\ell+1,\ldots, n_n$ of balls of increasing diameter connected by the function $F_1$. 
	\end{enumerate}
	The structure $\bK^{\cC}$ consequently defines an $n$-ultrametric space $\bU'$ bi-embeddable with $\bU_n$. This concludes Example~\ref{exa:ultrametric}.
\end{exa}

We now show that all these examples have finite big Ramsey degrees.

\begin{lemma}
    \label{Lem:Finite_KCD}
    Let $\bB$ is a finite $\bbL$-structure with $\cL_\bB$ finite. Let $\cC$ be a set of allowed vertex closures, $\cD$  a set of simple structures, and $\bA$  a finite structure. Then $\emb(\bA, \bB^{\cC, \cD})$ is finite.
\end{lemma}

\begin{proof}
    We assume $\emb(\bA, \bB^{\cC, \cD}) \neq \emptyset$. Given $v\in A$ and $f\in \emb(\bA, \bB^{\cC, \cD})$, consider $f(v):= (\cC, e)$. The functional part of $\cC$ is isomorphic to $\rm{Cl}_\bA(v)$, and the relational part of $\cC$ is isomorphic to a substructure of $\bB$, possibly with some relations removed. Hence there are only finitely many possibilities for $f(v)$.  
\end{proof}

\begin{prop}
    \label{Prop:Unary}
	Let $\bK$ be a $\bbL$-structure, $\cC$ a set of allowed vertex-closures and $\cD$ a set of simple substructures.
			If $\bA\in \age(\bK^{\cC,\cD})$ and $\cA$ is a finite set of finite substructures of $\bK$ such that for every $e\in \emb(\bA,\bK^{\cC,\cD})$
			there exists $\bB\in \mathcal A$ isomorphic to $\bK|_{\pi[e[A]]}$,
			then 
	$$\rm{BRD}(\bA, \bK^{\cC, \cD})\leq \sum_{\bB\in \cA} \rm{BRD}(\bB, \bK)\cdot\left|\left\{f\in \emb(\bA,\bB^{\cC,\cD}):\pi[f[A]]=B\right\}\right|.$$
\end{prop}

\begin{proof}
	Fix $\bK$, $\cC$, $\cD$,  $\bA$ and $\cA$ as in the statement.
	Consider a coloring $\gamma\colon \emb(\bA,\allowbreak \bK^{\cC,\cD})\to r$ for some $r<\omega$.  
	For every $\bB\in \cA$ we consider set $$S_\bB=\left\{f\in \emb(\bA,\bB^{\cC,\cD}):\pi[f[A]]=B\right\}  = \left\{f\in \emb(\bA,\bK^{\cC,\cD}):\pi[f[A]]=B\right\}.$$
	For every $e\in \emb(\bB,\bK)$ and $f\in S_\bB$ consider embedding $f^e\in \emb(\bA,\bK^{\cC,\cD})$ constructed as follows.
	Given $v\in A$ let $\bA'$ and $e'$ be such that $f(v)=(\bA',e')$ and put $f^e(v)=(\bA', e\circ e')$.
	Now we define coloring $\gamma_\bB:\emb(\bB,\bK)\to (S_\bB)^r$ by putting $\gamma_\bB(e)(f)=\gamma(f^e)$ for every $e\in \emb(\bB,\bK)$ and $f\in S_\bB$.

	Let $\varphi\in \emb(\bK,\bK)$ be such that the number of colors of every $\bB$ is at most $\rm{BRD}(\bB, \bK)$
	Construct $\varphi'\in \emb(\bK^{\cC,\cD}, \bK^{\cC,\cD})$ by putting $\varphi'((\bA',e'))=(\bA',\varphi\circ e')$.
\end{proof}

As an immediate corollary of Proposition~\ref{Prop:Unary}  we obtain that the structures discussed in Examples~\ref{exa:orientations}, \ref{exa:permutation}, \ref{exa:rado}, \ref{exa:S2} and \ref{exa:ultrametric} all have finite BRDs. With additional analysis it is possible to obtain corresponding big Ramsey structures.
We describe a general construction for the special case of \emph{unary} languages. That is, languages $\cL'\subseteq \bbL'$ where all
relation symbols (as well as all function symbols) are unary. This is the case of Examples~\ref{exa:orientations} and \ref{exa:ultrametric}.
We follow~\cite{Evans3} which describes the Ramsey expansions for free amalgamation classes in languages with functions symbols and adapt it to describe big Ramsey structures for \fr limits of amalgamation classes in unary languages, which can still be described in a relatively easy way.

Given a set of allowed closures $\cC$ and a structure $\bA$, we say that $\bA$ \emph{has closures in} $\cC$ if for every $v\in A$ there exists $\cB\in \cC$ describing $\mathrm{Cl}_\bA(v)$.
Notice that if $\cK$ is an amalgamation class of structures in a language $\cL'$ consisting of only unary relation and function symbols, then $\cK$ is a free amalgamation class and there exists a set of allowed closures $\cC$ defining the class $\cK$, in the sense that an  $\cL'$-structure $\bA$ is in $\cK$ if and only if it has closures in $\cC$.

Given an $\bbL'$-structure $\bK$ and $u,v\in K$ we write $u\preceq^\bK v$ if and only if $u\in \mathrm{Cl}_\bK(v)$.  We also write $u\sim^\bK v$ if $u\preceq^\bK v\preceq^\bK u$. It is easy to see that $\preceq^\bK$ is an preoder and $\sim^\bK$ is an equivalence relation. We call its equivalence classes \emph{closure-components}. 
Given $v\in K$ we denote by $[v]_\bK$ its closure-component.

If, for a structure $\bA$, there is $v\in \bA$ such that $\bA=\mathrm{Cl}_\bA(v)$, then $[v]_\bA$ is the set of all $u\in A$ with $\rm{Cl}_\bA(u) = \bA$, so in particular $[v]_\bA$ doesn't depend on the choice of generating element $v\in A$. We denote by $\bA^\circ\subseteq \bA$ the substructure (possibly empty) on underlying set $A\setminus [v]_\bA$; in particular, $A\setminus [v]_\bA$ is the domain of a substructure. 

We discuss some special orderings of structures. For this we fix a binary relation symbol $<\in \bbL$.  We call a structure $\bA$ \emph{ordered} if $<^\bA$ is a linear order, and we call $\bA$ \emph{unordered} if $<^\bA$ is $\emptyset$. If $\bA$ and $\bB$ are ordered structures, we say that $\bA$ is \emph{similar} to $\bB$ if there exists an isomorphism $f:\bA|_{\bbL'\setminus \{{<}\}}\to\bB|_{\bbL'\setminus \{{<}\}}$ such that $f|_{\bA^\circ}\colon \bA^\circ\to \bB^\circ$ is also an isomorphism (thus $f$ is order-preserving on all vertcies not in $f|_{\bA^\circ}$).
Given a structure in $\bK$ in a language not containing $<$ we call an expansion $\bK^*/\bK$ an \emph{order-expansion} if it only adds a binary relation $<^{\bK^*}$ which is a linear order.

\begin{defin}
Given a countable unordered structure $\bK$ and an order-expansion $\bK^*$ of $\bK$, we call $\bK^*/\bK$ \emph{closure-respecting} if $<^{\bK^*}$ satisfies:
\begin{enumerate}
  \item Every closure-component forms an interval.
  \item Whenever $u\not\sim^\bK v$ and $u\preceq^\bK v$ then also $u<^{\bK^*}v$.
  \item For every $u$ and $v$, if $\mathrm{Cl}_{\bK^*}(u)$ is similar to $\mathrm{Cl}_{\bK^*}(v)$,
	  then they are also isomorphic.
  \item If $K$ is infinite, then $<^{\bK^*}$  has order type $\omega$.
\end{enumerate}
\end{defin}

\begin{obs}
	Let $\bK$ be an unordered locally finite $\bbL'$-structure. Then every downset of $\prec_\bK$ and every closure-component is finite.
	Thus if $\bK$ is countable, then it has a closure-respecting order-expansion $\bK^*$.
\end{obs}

\begin{obs}
	\label{obs:extension}
	Let $\cL'\subseteq \bbL'$ be unary language.
	Let $\cK$ be a \fr class of unordered $\cL'$-structures and $\bK$ its \fr limit.
	Then every closure-respecting order-expansion $\bK^*$ satisfies the following form
	of the extension property:
	For every $v\in K$ it holds that there are infinitely many copies of
	$\mathrm{Cl}_{\bK^*}(v)$ extending $(\mathrm{Cl}_{\bK^*}(v))^\circ$.

	If $\bA^*$ is a substructure of $\bK^*$ and
	$f:\bA^*\to \bK^*$ is an injection which is an embedding on substructures generated by single vertices and
	preserves the relative order of closure-components, then $f$ is an embedding. 
\end{obs}

\begin{prop}
	\label{prop:recurence}
	Let $\cL'\subseteq \bbL'$ be unary language.
	Let $\cK$ be a \fr class of unordered $\cL'$-structures and $\bK$ its \fr limit.
	Then every closure-respecting expansion $\bK^*$ is recurrent.
\end{prop}
\begin{proof}
	Enumerate closure-components of $\bK^*$ as $C_1,C_2,\ldots$ in the order given by
	$(K,<^{\bK^*})$.  Given $\eta\in \emb(\bK)$, an embedding $\theta\in \emb(\bK)$ with $\eta\circ \theta\in \emb(\bK^*)$ can be constructed by induction using Observation~\ref{obs:extension}.
\end{proof}
\begin{theorem}
\label{thm:Unary}
Let $\cL'\subseteq \bbL'$ be unary language.
Let $\cK$ be a \fr class of unordered $\cL'$-structures and $\bK$ its \fr limit.
Then every closure-respecting expansion $\bK^*$ of $\bK$ is a big Ramsey structure.
\end{theorem}
\begin{proof}
	Fix $\cK$, $\bK$, and its closure-respecting expansion $\bK^*$.
	Without loss of generality assume that $\cL'\subseteq \bbL'\setminus
	\bbL$, $K=\omega$ and $<^{\bK^*}$ is the linear order (thus the language of $\bK^*$ is $\cL'\cup \{<\}\subseteq \bbL'$). 
	In particular we get that $\bM=\bK^*|_{\bbL}$ is $\omega$ with the natural ordering.

	By Proposition~\ref{prop:recurence} we only need to check that $\bK^*$ satisfies IRT which we will do using Theorem~\ref{Thm:Weak_Biembeddability}
	and Proposition~\ref{Prop:Unary}.
	Let $\cC$ be an inclussion minimal set of allowed closures so $\bK^*$ has closures in $\cC$.
	We construct a weak bi-embedding $(\bM^{\cC},\phi,\psi)$ for $\bK^*$ and show that $\bM^{\cC}$ satisfies IRT.

	First we describe the embedding $\phi\in \emb(\bK^*,\bM^{\cC})$. For every $v\in \bK^*$ we put $\phi(v)=(\bA_v,e_v)$ where $\bA_v$ is the unique structure in $\cC$ describing $\mathrm{Cl}_{\bK^*}(v)$ and $e_v$ is the (unique) isomorphism $e_v:\bA_v\to \mathrm{Cl}_{\bK^*}(v)$ with $e_v(0)=v$. Notice that $e_v$ is an embedding of the linear order $\bA|_\bbL\to \bM$, and by the construction of $\bM^{\cC}$, it follows that $\phi$ is an embedding $\bK^*\to \bM^{\cC}$.

	Next we describe $\psi\in \emb_{\bK^*}(\bM^\cC, \bK^*)$.
	By the construction of $\bM^{\cC}$, for every $v\in \bM$, there are only finitely many closure-components of $\bM^{\cC}$ with $\max \pi(C_j)=v$. We can thus fix an enumeration of the closure-components of $\bM^{\cC}$ as $C_1,C_2,\ldots$ such that for every $i<j$, we have $\max \pi(C_i)\leq \max \pi(C_j)$.
	By repeated application of the extension property given by Observation~\ref{obs:extension},
	construct a $\cK$-approximate embedding $\psi$ by induction on this enumeration so that $\psi$ is an embedding $\bM^{\cC}|_{\cL}\to \bK$ and such that every closure-component forms an interval, with  the relative order of these intervals preserved.  By the second part of Observation~\ref{obs:extension}, this yields a weak bi-embedding.

	It remains to verify that $\bM^C$ satisfies IRT.
	By the Ramsey theorem applied for $\bA|_\bbL$ (which is just finite linear order of size $|A|$) and $\bM=\omega$ we obtain that $\rm{BRD}(\bA|_\bbL,\bM)=1$.
	By Proposition~\ref{Prop:Unary} for $\bA$ and $\cA=\{\bA|_\bbL\}$ we obtain that $\rm{BRD}(\bA,\bM^{\cC})=1$.
	Using the weak bi-embedding and Theorem~\ref{Thm:Weak_Biembeddability} we get IRT for $\bK^*$.
\end{proof}

Examples of structures which can be interpreted in a language with relational and unary function symbols include:
\begin{enumerate}
	\item $\Lambda$-ultrametric spaces (with small Ramsey properties studied in \cite{braunfeld2017ramsey}),
	\item Graphs which admits a $k$-orientation, i.e.\ an orientation of edges with out-degree at most $k$. The small Ramsey properties of these classes are studied in \cite{evans2019automorphism}.
\end{enumerate}
Additional examples are indicated in the following section.

\section{Catalog of known big Ramsey structures}
The following catalog gives what we believe to be a complete list of big Ramsey structures described to the date (or which directly follow from the framework developed in this paper). We also include a (partial) list of examples of structures with known upper bounds on big Ramsey degrees for which the exact big Ramsey degrees and/or the big Ramsey structures can hopefully be identified in future.

\def\pfR{\textbf R}
\def\pfM{\textbf M}
\def\pfPM{\textbf {PM}}
\def\pfCT{\textbf {CT}}
\def\pfCS{\textbf {CS}}
\def\pfS{\textbf S}
\def\pfC{\textbf C}
\def\pfU{\textbf 1}
We indicate the basic proof technique:
\begin{enumerate}
	\item[\pfR]  Ramsey theorem,
  \item[\pfM]  Milliken tree theorem,
  \item[\pfPM] product form Milliken tree theorem with unbunded branching,
  \item[\pfCT] coding trees with proof based on language of set-theoretic forcing theorem,
  \item[\pfCS] Carlson-Simpson theorem,
  \item[\pfS] Ramsey theorem for trees with sucessor operation,
  \item[\pfC] custom Ramsey argument,
  \item[\pfU{}] indicates that the underlying Ramsey theorem is used in combination with adding unary relations and functions as discussed in Section~\ref{Sec:Functions}.
\end{enumerate}
In the table some cases are generalizations of others. We list results only in the most general case they cover.
$k$ is used to denote an arbitrary fixed integer.

\def\wname{0.22\textwidth}
\def\wspecial{0.16\textwidth}
\def\wfiniteBRD{0.16\textwidth}
\def\wBRD{0.16\textwidth}
\def\wBRS{0.16\textwidth}

\def\ename#1{\parbox{\wname}{\raggedright\vskip1mm #1 \vskip1mm}}
\def\especial#1{&\parbox{\wspecial}{\raggedright\vskip1mm #1 \vskip1mm}}
\def\efiniteBRD#1{&\parbox{\wfiniteBRD}{\raggedright\vskip1mm #1 \vskip1mm}}
\def\eBRD#1{&\parbox{\wBRD}{\raggedright\vskip1mm #1 \vskip1mm}}
\def\eBRS#1{&\parbox{\wBRD}{\raggedright\vskip1mm #1 \vskip1mm}\\\hline}
{
\small
\begin{longtable}[t]{|p{\wname}|p{\wspecial}|p{\wfiniteBRD}|p{\wBRD}|p{\wBRS}|}
	\hline 
	\ename{\textbf{structure}}
	 \especial{\textbf {special cases}}
	 \efiniteBRD{\textbf{finite BRDs}}
	 \eBRD{\textbf{precise BRDs}}
	 \eBRS{\textbf{big\\ Ramsey structure} }

	\ename{$(\omega,<)$}
	 \especial{pigeonhole}
	 \efiniteBRD{Ramsey thm.~\cite{Ramsey} \pfR}
	 \eBRD{Ramsey thm.~\cite{Ramsey} \pfR}
	 \eBRS{trivial}

	\ename{$(\mathbb Q,<)$}
	 \especial{vertices (easy);\\pairs~\cite{Galvin}}
	 \efiniteBRD{Laver \pfC\\(unpublished, see \cite{Stevo_book})}
	 \eBRD{\cite{Devlin} \pfM}
	 \eBRS{\cite{Zucker_BR_Dynamics}\\ Ex.~\ref{Exa:Rationals} \pfM}

	\ename{Rado graph}
	 \especial{vertices (folklore);\\edges~\cite{PS_1996}}
	 \efiniteBRD{\cite{Sauer_2006}~\pfM}
	 \eBRD{\cite{Laflamme_S_V}~\pfM}
	 \eBRS{\cite{Zucker_BR_Dynamics}\\Ex.~\ref{Exa:Rado} \pfM}

%

	\ename{Unrestricted; finite binary language\footnote{Special cases include the generic tournament and the generic directed graph}}
	  \especial{vertices (folklore)}
	  \efiniteBRD{\cite{Laflamme_S_V,dobrinen2016rainbow}~\pfM}
	  \eBRD{ \cite{Laflamme_S_V,dobrinen2016rainbow}~\pfM}
	  \eBRS{\cite{Zucker_BR_Dynamics}~\pfM}

	\ename{Ultrametric spaces}
	  \especial{vertices  \cite{delhomme2008indivisible}}
	  \efiniteBRD{\cite{NVT2008} \pfR; Ex.~\ref{exa:ultrametric} \pfR\pfU}
	  \eBRD{\cite{NVT2008} \pfR; Thm.~\ref{thm:Unary}~\pfR\pfU}
	  \eBRS{\cite{NVT2008} (implicit); Thm.~\ref{thm:Unary}~\pfR\pfU}

	\ename{Dense local order $\bS(2)$}
	  \especial{}
	  \efiniteBRD{\cite{L_NVT_S_2010} \pfC;\\ Ex.~\ref{exa:S2} \pfM\pfU}
	  \eBRD{\cite{L_NVT_S_2010} \pfC;\\ \cite{CDP_SDAP2} \pfCT}
	  \eBRS{\cite{Zucker_BR_Dynamics} \pfC\\}

	\ename{$\bS(k)$, $k>2$}
	  \especial{}
	  \efiniteBRD{\cite{DB} \pfC; Ex.~\ref{exa:S2} \pfM\pfU}
	  \eBRD{\cite{DB} \pfC}
	  \eBRS{\cite{DB} (implicit) \pfC}

	\ename{$(\mathbb Q,<)$ partitioned to $n$ dense parts}
	  \especial{}
	  \efiniteBRD{\cite{L_NVT_S_2010} \pfC;\\ Ex.~\ref{exa:S2} \pfM\pfU}
	  \eBRD{\cite{L_NVT_S_2010} \pfC;\\ \cite{CDP_SDAP2} \pfCT}
	  \eBRS{\cite{Zucker_BR_Dynamics} \pfC}

	\ename{\fr limit of convexly ordered eqivalences with $k$ classes}
	  \especial{}
	  \efiniteBRD{\cite{CDP_SDAP2} \pfCT; Ex.~\ref{exa:S2} \pfM\pfU}
	  \eBRD{\cite{CDP_SDAP2} \pfCT}
	  \eBRS{\cite{CDP_SDAP2} (implicit)}

	\ename{$\mathbb Q_\mathbb Q$, $\mathbb Q_{\mathbb Q_\mathbb Q}$, \ldots}
	  \especial{}
	  \efiniteBRD{\cite{CDP_SDAP2} \pfCT}
	  \eBRD{\cite{CDP_SDAP2} \pfCT}
	  \eBRS{\cite{CDP_SDAP2} (implicit)}

	\ename{Universal $k$-partite graph}
	  \especial{}
	  \efiniteBRD{\cite{Zucker_BR_Upper_Bound,CDP_SDAP2}\footnote{Finite big Ramsey degrees for bipartite graphs was also announced by John Howe} \pfCT; Ex. \ref{exa:rado} \pfM\pfU}
	  \eBRD{\cite{BCDHKVZ,CDP_SDAP2} \pfCT; Ex. \ref{exa:rado} \pfM\pfU}
	  \eBRS{\cite{CDP_SDAP2} \pfCT}
	\ename{Ordered universal $k$-partite graph}
	  \especial{}
	  \efiniteBRD{\cite{CDP_SDAP2} \pfCT}
	  \eBRD{\cite{CDP_SDAP2} \pfCT}
	  \eBRS{\cite{CDP_SDAP2} \pfCT}

	\ename{Universal ($\omega$-partite) graph
	with labeled partitions}
	  \especial{}
	  \efiniteBRD{\cite{BCdRHKK} \pfPM;\\ Ex. \ref{exa:rado}}
	  \eBRD{}
	  \eBRS{See discussion below}

	\ename{Structures in unary languages (possibly with functions)}
	 \especial{}
	 \efiniteBRD{\cite{hubivcka2019all,Evans3}~\pfR\pfU}
	 \eBRD{Thm.~\ref{thm:Unary}~\pfR\pfU}
	 \eBRS{Thm.~\ref{thm:Unary}~\pfR\pfU}

	\ename{Universal triangle-free graph}
	  \especial{vertices \cite{komjath1986,Elzahar1989};\\ Edges \cite{sauer1998}}
	  \efiniteBRD{\cite{Dobrinen_3Free} \pfCT;\\\cite{Hubicka2020CS} \pfCS}
	  \eBRD{\cite{BCDHKVZ,dobrinen2020ramsey} \pfCT;\\ \cite{BCDHKVZ_poset} \pfC}
	  \eBRS{\cite{BCDHKVZ,BCDHKVZ_poset,DZ} \pfCT}

	\ename{Universal $K_k$-free (Henson) graph}
	  \especial{vertices \cite{komjath1986,Elzahar1989}}
	  \efiniteBRD{\cite{Dobrinen_Henson,DZ} \pfCT}
	  \eBRD{\cite{BCDHKVZ,DZ,Vodsedalek2025} \pfCT}
	  \eBRS{\cite{BCDHKVZ,DZ} \pfCT}

		\ename{\fr limit of a finitely constrained free amalgamation class in a binary language}
	  \especial{vertices\\\cite{sauer2003canonical,el1993divisibility}\\(special cases)}
	  \efiniteBRD{\cite{BCDHKVZ,DZ} \pfCT;\\\cite{Balko2023Sucessor} \pfS}
	  \eBRD{\cite{BCDHKVZ,DZ} \pfCT}
	  \eBRS{\cite{BCDHKVZ,DZ} \pfCT}


	\ename{Universal poset}
	  \especial{}
	  \efiniteBRD{\cite{Hubicka2020CS}~\pfCS}
	  \eBRD{\cite{BCDHKVZ_poset}~\pfC{} or \pfS}
	  \eBRS{\cite{BCDHKVZ_poset}~\pfC{} or \pfS}

	\ename{free superpositions of multiple linear orders}
	  \especial{}
	  \efiniteBRD{\cite{Hubicka2020CS} \pfCS;\\
      Ex. \ref{exa:2orders} \pfM
      \\ Ex. \ref{exa:permutation} \pfM \pfU}
	  \eBRD{Ex. \ref{exa:2orders} \pfM}
	  \eBRS{Ex. \ref{exa:2orders} \pfM}

	\ename{Metric spaces definable in a free superposition of ultrametrics and random graphs}
	  \especial{vertices (easy)}
	  \efiniteBRD{\cite{Masulovic_2020} \pfC;\\ \cite{balko2021big} \pfCS}
	  \eBRD{?}
	  \eBRS{?}

	\ename{Metric spaces\\ with distances \{0,1,\ldots, k\}}
	  \especial{vertices \cite{delhomme2008indivisible}}
	  \efiniteBRD{\cite{Hubicka2020CS,balko2021big} \pfCS}
	  \eBRD{?}
	  \eBRS{?}

	\ename{Universal 3-uniform hypergraph}
	  \especial{vertices \cite{el1994divisibility}}
	  \efiniteBRD{\cite{BCHKV_3Unif} \pfPM}
	  \eBRD{?}
	  \eBRS{?}

	\ename{Unrestricted structures in injective
	languages having, $\forall n>1$, finitely many relations of arity $n$}
	  \especial{}
	  \efiniteBRD{\cite{BCdRHKK} \pfPM}
	  \eBRD{?}
	  \eBRS{?}
\end{longtable}
}

As promised in the table, we briefly discuss the case of a Rado graph with a generic partition into countably many labeled pieces (the case of the $\omega$-partite Rado graph is similar). A big Ramsey structure can be formed by generalizing the notion of an LSV tree. In the case of a partition into finitely many labeled pieces, this is straightforward; one modifies the definition of an LSV coding tree to be a subtree of $n\times 2^{<\omega}$, where $n$ is the number of pieces of the partition, but where each level has exactly one splitting node or exactly one coding node as before. For a partition into countably many pieces, it is more complicated. Localizing to any finite number of pieces, we obtain something resembling an LSV coding subtree of $n\times 2^{<\omega}$, but globally, the set of ``interesting levels,'' i.e.\ levels which contain a splitting node or a coding node, is no longer ordered in order type $\omega$. It is open if the Rado graph with a countable labeled partition admits a recurrent big Ramsey structure.  

It is a relatively direct consequence of the product Ramsey theorem that given
two Ramsey classes $\cK$ and $\cK'$, the age of the structure created as a free
superposition of their respective \fr limits is also Ramsey, see
e.g.~\cite{Bodirsky2015,hubicka2025twenty}. While the proof of this property does not naturally generalize to big
Ramsey degrees, in special cases such results can be established.  Notice that
the Ramsey theorem \textbf{R} is implied by the Milliken tree theorem (\textbf{M} or \pfPM)
as well as by the Carlson-Simpson theorem \pfCS.  The Milliken tree theorem \textbf{M},
certain cases of \textbf{PM} (as used in the proofs of upper bounds on big Ramsey degrees), and the
Carlson-Simpson theorem \textbf{CS} are all implied by the Ramsey theorem for trees with
successor operation \textbf{S}, which is well-behaved for products.  It
follows that upper bounds on big Ramsey degrees of structures proved by those
theorems can be combined to give upper bounds on big Ramsey degrees of structures
created as free superpositions of those.
This yields additional examples of structures with finite big Ramsey degrees,
such as the $\omega$-ordered variants of structures above discussed in~\cite{Masulovic_2020}.
Curiously enough, the precise big Ramsey degrees does not translate so directly, and additional analysis
is required to understand the exact big Ramsey degrees of a free superposition even if the big Ramsey
degrees of the original structures are fully characterised.

\section{Negative results}
The classification programme of Ramsey
classes~\cite{Nevsetril2005,hubicka2025twenty,hubivcka2019all} provides a rich
source of \fr limits which are good candidates for structures with bounded big
Ramsey degrees. In addition to the positive examples of big Ramsey structures discussed here, there are multiple techniques which can be used to \emph{disprove}
finiteness of big Ramsey degrees.  
We quickly note a few negative results on big Ramsey degrees in the literature, which
can be broadly divided into the following types:
\begin{enumerate}
	\item counting number of oscillations of monotone functions assigned to sub-objects~\cite{ChEW_pseudotree, Todorcevic_Roscillate, Bartosova2025},
	\item study of the partial order of ages (ranks or orbits) of vertices~\cite{sauer2003canonical}, 
	\item arguments based on the distance and diameter in metric spaces~\cite{Laflamme2006}.
	\item arguments based on compressing tree of types that is infinitely branching on infinitely many levels~\cite{Omegalabelled2025}.
\end{enumerate}
In all known negative results, something stronger is proven: Given an infinite structure $\bK$ and some $\bA\in \age(\bK)$ with $\rm{BRD}(\bA, \bK) = \infty$, one typically proves this by constructing an unavoidable coloring $\gamma\colon \emb(\bA, \bK)\to \omega$. In principle, one should only need to construct unavoidable colorings $\gamma\colon \emb(\bA, \bK)\to n$ for every $n< \omega$. It would be interesting to find an example with $\rm{BRD}(\bA, \bK) = \infty$, but where there is no unavoidable $\omega$-coloring. This question is an instance of a much more general question: For infinite structures which do not have finite big Ramsey degrees, what (if anything) is the correct analogue of a big Ramsey structure?

\vspace{2 mm}

\noindent
\textbf{Acknowledgements:} We would like to thank Natasha Dobrinen and Mat\v{e}j Kone\v{c}\-\allowbreak n\'y for useful comments on an early draft and for several useful discussions.

\noindent
\textbf{Funding:} J.H. was supported by the project 21-10775S of  the  Czech  Science Foundation (GA\v CR) and in later stages by the
European Research Council (ERC Synergy Grant 810115 Dynasnet).  A.Z.\ was supported by NSERC grants RGPIN-2023-03269 and DGECR-2023-00412.

\bibliographystyle{amsplain}
\bibliography{bibzucker}

\end{document}